\documentclass[12pt]{amsart}
\usepackage{url,amstext,amsfonts,amssymb,amscd,amsbsy,amsmath,amsthm,verbatim, mathrsfs, cite, hyperref}
\usepackage{appendix}
\usepackage{color}

\usepackage[all,pdftex,arc,curve,color,frame,poly,2cell]{xy}
\usepackage{graphicx}

\usepackage{ytableau} 

\usepackage{geometry}                
\newtheorem{theorem}{Theorem}[section]
\newtheorem{defn}[theorem]{Definition}
\newtheorem{conjecture}[theorem]{Conjecture}
\newtheorem{prop}[theorem]{Proposition}

\newtheorem{lemma}[theorem]{Lemma}

\theoremstyle{remark}
\newtheorem{remark}[theorem]{Remark}
\newtheorem{example}[theorem]{Example}
\theoremstyle{definition}

\newcommand{\defi}[1]{\textsf{#1}} 

\newcommand{\Gr}{\operatorname{Gr}}

\newcommand{\GL}{\operatorname{GL}}

\def\o{{\otimes }}
\def\u#1{{\underline{#1} }}
\def\w#1{{\widehat{#1} }}
\def\R{{\mathcal R}}
\def\B{{\mathcal B}}

\def\Q{{\mathcal Q}}
\def\V{{\mathcal V}}
\def\W{{\mathcal W}}
\def\O{{\mathcal O}}

\def\I{{\mathcal I}}
\def\T{{\mathcal T}}
\def\M{{\mathcal M}}

\def\E{{\mathcal E}}

\newcommand{\CC}{\mathbb{C}}

\newcommand{\PP}{\mathbb{P}}

\newcommand{\ZZ}{\mathbb{Z}}

\def\AA{{A_{1}^{*}\o\dots\o A_{n}^{*}}}
\def\AAn{{A_{1}\o\dots\o A_{n}}}
\def\AAp{{A'^{*}_{1}\o\dots\o A'^{*}_{n}}}
\def\xx{{x_{1}\o\dots\o x_{n}}}

\def\tA{{A^{*}_{1}\times\dots\times A^{*}_{n}}}
\def\PPA{{\PP A^{*}_{1}\times\dots\times \PP A^{*}_{n}}}
\def\Rn{{\R^{*}_{1}\otimes\dots\otimes \R^{*}_{n}}}
\def\PPAp{{\PP A'^{*}_{1}\times\dots\times \PP A'^{*}_{n}}}

\def\GLA{{\GL(A_{1})\times\dots\times\GL(A_{n})}}
\def\GLAp{{\GL(A'_{1})\times\dots\times\GL(A'_{n})}}

\def\Gran {{\Gr(r_{1},A_{1}^{*})\times \dots \times \Gr(r_{n},A_{n}^{*})}}
\def\SubA{{\Sub_{r_{1},\dots,r_{n}}(A_{1}^{*}\o\dots\o A_{n}^{*})}}

\def\o{{\otimes }}
\def\u#1{{\underline{#1} }}
\def\w#1{{\widehat{#1} }}

\newcommand{\Sym}{\operatorname{Sym}}

\def\o{{\otimes}}
\def\bw#1{{\textstyle\bigwedge^{\hspace{-.2em}#1}}}

\newcommand{\codim}{\operatorname{codim}}
\newcommand{\depth}{\operatorname{depth}}

\newcommand{\rank}{\operatorname{Rank}}

\newcommand{\Seg}{\operatorname{Seg}}
\newcommand{\Sub}{\operatorname{Sub}}

\def\red#1{{\textcolor{red}{#1}}}
\def\blue#1{{\textcolor{blue}{#1}}}

\begin{document}
\author[Oeding]{Luke Oeding}
\address{Department of Mathematics, Auburn University}
\date{\today}
\title[aCM Secants]{Are all Secant Varieties of Segre Products \\ Arithmetically Cohen-Macaulay?}
\begin{abstract}
When present, the Cohen-Macaulay property can be useful for finding the minimal defining equations of an algebraic variety.  It is conjectured that all secant varieties of Segre products of projective spaces are arithmetically Cohen-Macaulay. A summary of the known cases where the conjecture is true is given. An inductive procedure based on the work of Landsberg and Weyman (LW-lifting) is described and used to obtain resolutions of orbits of secant varieties from those of smaller secant varieties.   A new computation of the minimal free resolution of rank 4 tensors of format $3 \times 3 \times 4$ is given. LW-lifting is used to prove several cases where secant varieties are arithmetically Cohen-Macaulay and arithmetically Gorenstein. 
\end{abstract}
\maketitle

\section{Introduction}
Implicitization problems are central in Applied Algebraic Geometry.  Starting, for instance, with an algebraic-statistical model for structured data (such as tensors with low rank) we often ask for the implicit defining equations for the associated algebraic variety. These equations might be used, for instance, for testing whether a data point is on the given model.  Usually some of these equations can be found (for example by linear algebra, ad hoc methods, or analyzing symmetry). A difficult problem is then to determine when the known equations suffice.  Algebraic properties such as the arithmetically Cohen-Macaulay (aCM) property (see Definition \ref{def:aCM}) can be quite useful for this, if they can be determined.

This note focuses on tensors of restricted border rank, or secant varieties of Segre products, denoted $\sigma_{r}(\Seg (\PP^{n_{1}-1} \times \dots\times \PP^{n_{d}-1})  )$, which consist of the (Zariski closure of the)  tensors of format $n_{1}\times\dots\times n_{d}$ and rank $r$.
 Our main focus is on the following.
\begin{conjecture}\label{conj:aCM}
 All secant varieties of Segre product of projective spaces are arithmetically Cohen-Macaulay.
\end{conjecture}

In Section~\ref{sec:NotationEvidence} we set our notation and collect all the evidence for Conjecture~\ref{conj:aCM}. 

A standard approach to proving such a conjecture about a family that depends on several integer parameters would be to apply a multi-step induction. The base case $r=1$ and $n_{i}$ all arbitrary is true because it is well known that homogeneous varieties (of which the Segre variety is one) are aCM. If we knew, for instance, that secants of aCM varieties were also aCM, then we would be done. At present we don't know if this statement is valid and if so, how to prove it. So we consider a more sophisticated induction procedure.
Landsberg and Weyman's method in \cite{Landsberg-Weyman-Bull07} is the first step in this direction: under certain technical hypotheses (resolutions by small partitions) and for $n_{i}\geq r$ for all $i$, the aCM property is inherited.
\begin{theorem}[Landsberg-Weyman \cite{Landsberg-Weyman-Bull07} ]
Suppose  $\sigma_{r}(\Seg (\PP^{r-1} \times \dots\times \PP^{r-1})  )$ is aCM and has a resolution by small partitions. Then  $\sigma_{r}(\Seg (\PP^{n_{1}-1} \times \dots\times \PP^{n_{d}-1})  )$  is aCM for all $n_{i}\geq r$. 
\end{theorem}

From this we are naturally led to ask two questions.  (1) What does the technical hypothesis ``resolution by small partitions,''  mean, and is it satisfied for the varieties we're interested in? (2) Can we adapt LW-lifting to the cases when $n_{i}<r$ for some $i$? 
In the appendices we repeat Landsberg and Weyman's argument almost verbatim but adapted to the smaller dimensional cases to prove the best result we can using this technique in the cases $n_{i}<r$, partially answering question (2), and explain the appearance of the curious ``resolution by small partitions.''

An $R$-module $M$ has a resolution by small partitions if the Schur modules which occur in the $G$-equivariant resolution of $M$ are indexed by partitions that fit inside prescribed sized boxes. 
More specifically, we will say that a $G$-variety $Y$ has an \emph{$(s_{j})$-small resolution} if its coordinate ring has a free resolution with the property that every Schur module $S_{\pi}A$ occurring in the resolution satisfies the following condition
\[\text{for each $j$ the first part of $\pi^{j}$ is not greater than $s_{j}$.}\]  

Suppose vector spaces $A_{i}' \subseteq A_{i}$ for $1\leq i \leq n$.
Let $\w{a_{j}}:=\frac{a_{1}\dots a_{n}}{a_{j}}$,  $\w{r_{j}} :=\frac{r_{1}\dots r_{n}}{r_{j}}$, $G = \GLA$ and $G' = \GLAp$.
Now we can state our result:
\begin{theorem}[(adapted from \cite{Landsberg-Weyman-Bull07})]\label{thm:main}
If a $G'$-variety $Y$ is aCM with a resolution 
 that is $(\w{r_{j}} -r_{j})$-small 
 for every $j$ for which $0<r_{j} <a_{j}$,
then $\overline{G.Y}$ is aCM.

Moreover we obtain a (not necessarily minimal) resolution of $\overline{G.Y}$ that is $(s_{j})$-small with

\begin{center} $s_{j} = 
\max_{\pi}
\begin{cases}
 \w{a_{j}}-r_{j}, & \text{ if } r_{j}<a_{j} \\
 \w{a_{j}} - \w{r_{j}}+\pi^{j}_{1}, & \text{ if }r_{j}=a_{j},
\end{cases}$\end{center}

where the $\max$ is taken over all multi-partitions $\pi$ occurring in the resolution of $\CC[Y]$.
\end{theorem}
Our proof of this theorem is quite involved, but is carried over almost verbatim from Landsberg and Weyman's work. We include the details in the Appendices so that we can get the more refined version of the result.  
In Appendix \ref{sec:orbits} we describe the relevant algebraic varieties as orbits of smaller varieties. 
In Appendix \ref{sec:weyman} we recall Weyman's geometric technique combined with a partial desingularization of orbit closures by subspace varieties.
In Appendix \ref{sec:cone} we describe an iterated mapping cone construction, which uses the calculations of cohomology of vector bundles aided by Bott's algorithm and the Borel-Weil theorem described in \ref{sec:Bott}.

In Section~\ref{newACM} we apply this result to several situations yielding many new families of secant varieties of Segre products for which the aCM Conjecture holds. 
In Section~\ref{newAG} we apply this result to determine new cases where secant varieties of Segre products are arithmetically Gorenstein.

\subsection{First Questions}
Recall that if $X \subset \PP^{N}$ is an algebraic variety, its secant varieties $\sigma_{r}(X)\subset \PP^{N}$ are defined as the Zariski closure of all points of the form $[x_{1}+\cdots+ x_{r}]$ with all $x_{i}\in X$.  Secant varieties arise in many different contexts. Perhaps the first place is in classical algebraic geometry. One may ask when can a given projective variety $X\subset \PP^{n}$ be isomorphically projected into $\PP^{n-1}$?
The answer is determined by the dimension of the secant variety $\sigma_{2}(X)$.  If $\sigma_{2}(X)$ does not fill the ambient $\PP^{N}$, then there is a point outside $X$ from which one can project $X$ into a $\PP^{N-1}$, \cite{Harris, DolgachevAG}. 

There are many modern applications of tensors and secant varieties, such as Geometric Complexity Theory \cite{LandsbergTensorBook, LandsbergSummary}, Algebraic Statistics \cite{Pachter-Sturmfels, AllmanRhodes08}, Phylogenetics and the so-called \emph{salmon conjecture}, \cite{prize, Friedland2010_salmon, BatesOeding, Friedland-Gross2011_salmon}, Signal Processing \cite{SahnounComon, deSilva_Lim_ill_posed, lathauwer:642_matrix_diag, OedOtt13_Waring}, and Coding Theory \cite{AOP_Grassmann}. 

What is the dimension of $\sigma_{r}(X)$ and for which $r$ does $\sigma_{r}(X)$ fill the ambient $\PP \CC^{N}$? While an expected answer can be easily calculated by dimension counting, the actual dimension might be less than expected. 

For Veronese varieties all dimensions of all higher secant varieties are known, thanks to the work of Alexander and Hirschowitz, \cite{AH95}. See also  Ottaviani and Brambilla's nice exposition \cite{OttBra08_AH}. 
Regarding secant varieties of Segre products, their dimensions have been widely studied, see \cite{ChiCil06_ksecant, Ballico2005_weak, CGG_P1s, CGG4_Segre, AOP_Segre} for some highlights, however the analogue to the Alexander-Hirschowitz result is not yet complete.

Perhaps the next question is to find polynomial defining equations of $\sigma_{r}(X)$. This is useful, in particular for membership testing, and has also been widely studied in the Segre case:
\cite{Ott07_Luroth,Landsberg-Manivel04,
Landsberg-Weyman-Bull07, RaicuGSS, SidmanSullivant_Prolongations, LanWey_secant_CHSS, LanOtt11_Equations, RaicuCat, Strassen83_rank, CEO, LanMan08_Strassen,
Bernardi_Ideal_Sym, Kanev_catalecticant,
RaicuGSS, Raicu_thesis, SidmanVermeire_Secant}.

There is also much interest in actually finding a decomposition: Given $\T \in \CC^{N}$, can you find an expression of $\T $ as a sum of points from $X$? For recent progress, see  \cite{Nie_STD, OedOtt13_Waring, CGLM_Waring, JMLR:v15:anandkumar14a}.

One may also ask if \emph{generic (respectively specific) identifiability} holds. That is, does a general (resp. specific) $\T \in \CC^{N}$ have an essentially unique (up to reordering and re-scaling) decomposition?  See  \cite{ChiOttVan, BocciChiantini, BocciChiantiniOttaviani, HOOS} for the state of the art on these topics.

Knowing equations of secant varieties can help with all of these questions, especially if they're determinantal.
Often some equations for secant varieties are known, but the difficult question is to show when the known equations suffice. Knowing whether the aCM property holds for all secant varieties of Segre products could help all of these questions.


\subsection{Questions about primeness and degrees of minimal generators}

In general, given a parametrized (irreducible) variety $X\subset \PP^{N}$, suppose we have found candidate minimal generators  $f_{1},\ldots,f_{t}$.  Set $J :=\langle f_{1},\ldots,f_{t}\rangle$. The question then remains: What is the maximal degree of minimal defining equations of a given secant variety, and when do the known equations suffice? This question is well studied in some cases such as monomial ideals \cite{SturmfelsSullivant}, for curves \cite{Ginensky, SidmanVermeire_Secant}, and in some infinite dimensional cases, however the general question is still very open.

It may be possible to obtain upper bounds (via Castelenuovo-Mumford regularity, for example \cite{Eisenbud_tome, Eisenbud_syzygies, MaclaganSmith}), but these computations are often also difficult and the upper bounds obtained may not be sharp.  Another approach undertaken by Aschenbrenner and Hillar \cite{AshHillarFinite},  Draisma and Kutler \cite{draisma-kuttler_bounded}, Sam and Snowden \cite{SamSnowden,sam-snowden-tca, SamBounded}, and others is to investigate these questions in an infinite dimensional setting. This method has been used to determine when certain ideals are ``Noetherian up to symmetry'' and in turn, this can sometimes be used to provide a non-constructive guarantee that tensors of bounded rank are defined by equations in bounded degree not depending on the number of tensor factors. This method, however, does not typically give an explicit bound.

Suppose that we have also shown that $\V(J) = X$ as a set. Sometimes this can inferred from information provided by a numerical irreducible decomposition in  \defi{Bertini}\cite{Bertini}, which will tell the number of components in each dimension together with their geometric degrees. A symbolic degree computation can often indicate that $J$ is reduced in its top dimension if the symbolic and numerical results agree. A particularly challenging step is to determine if indeed $J  = \I(X)$, because perhaps there are lower dimensional embedded primes.

So we might attempt to show that $J$ is prime. The set-theoretic result and the fact that $I(X)$ is prime then would imply that $J = I(X)$. 
Sometimes symmetry and knowing a list of orbits can provide enough information about the primary decomposition of $J$  to rule out embedded primes (see \cite{AholtOeding}).

Showing that a given ideal is prime is one of the most basic questions in algebra.
Another way to know when the given equations generate a prime ideal that might be available is if the variety of study is arithmetically Cohen-Macaulay (aCM) (see also \cite{Trung_degree_bounds}).
\begin{defn}\label{def:aCM}
Suppose $R$ is a polynomial ring in finitely many variables, and let $I$ be an ideal of $R$. Then $R/I$ is \emph{Cohen-Macaulay} if $\depth  R/I = \codim I$. We say that $X = \V(I)$ is arithmetically Cohen-Macaulay (aCM) if $R/I$ is Cohen-Macaulay. By the Auslander-Buchsbaum formula, $R/I$ is CM if and only if the projective dimension of $R/I$ (the minimal length of a free resolution of $R/I$) is equal to the codimension of $I$.
\end{defn}
A standard argument to show primeness is the following:  If an ideal $J$ in a polynomial ring $R$ is aCM and the affine scheme it defines is generically reduced, then it is everywhere reduced, and if the zero set $\V(J)$ agrees with $X$, then $J = \I(X)$.  This technique is used in \cite{Landsberg-Weyman-Bull07} in the Segre case and  for secants of compact Hermitian symmetric spaces (CHSS) in \cite{LanWey_secant_CHSS}.

\section{Secant varieties and tensors}\label{sec:NotationEvidence}
Let $A_{1},\ldots,A_{d} $, be $\CC$-vector spaces, then the tensor product $A_{1}\o\dots\o A_{n}$ is the vector space with elements $(T_{i_{1},\ldots,i_{d}})$ considered  as hyper-matrices or tensors.

The \emph{Segre variety} of rank 1 tensors is defined by
\begin{eqnarray*} \Seg: \PP A_{1} \times \dots \times \PP A_{d} &\longrightarrow& \PP \big(A_{1}\otimes \dots \otimes A_{d} \big) \\
 ([v_{1}],\dots,[v_{d}]) &\longmapsto & [v_{1}\o \cdots \o v_{d}]
.\end{eqnarray*}
In coordinates points on the Segre variety have the form  $T_{i_{1},\ldots,i_{d}} = v_{1,i_{1}}\cdot v_{2,i_{2}}\cdots v_{d,i_{d}}$.
The $r^{th}$ \emph{secant variety}  of a variety $X \subset \PP^{N}$ is
\[
\sigma_{r}(X) := \overline{\bigcup_{x_{1},\dots,x_{r} \in X} \PP( \text{span}\{x_{1},\dots,x_{r}\})}\subset \PP ^{N}
.\]
General points of $\sigma_{r}(\Seg ( \PP A_{1} \times \cdots \times \PP A_{d}))$ have the form 
$\left[\sum_{s=1}^{r}v_{1}^{s}\o v_{2}^{s}\o\dots \o  v_{d}^{s}\right],$
or
in coordinates: $T_{i_{1},\ldots,i_{n}} = \sum_{s=1}^{r} v_{1,i_{1}}^{s}\cdot v_{2,i_{2}}^{s}\cdots v_{d,i_{d}}^{s}$.

Here are the cases we know  Conjecture~\ref{conj:aCM} to be true:
\begin{itemize}
\item \blue{Segre varieties}: $X = \Seg ( \PP A_{1} \times \cdots \times \PP A_{d}) $ are homogeneous and thus aCM.

Straightforward proof: The coordinate ring of $X$ is $\bigoplus_{d}H^{0}(\PP^{N},\O(1,\dots,1)(-d))$.  One notices (via Bott's algorithm and the Borel-Weil theorem) that the structure sheaf of $X$ has no intermediate cohomology, and thus the coordinate ring is aCM.

\item \blue{Ambient spaces}: If $k$ is such that $\sigma_{k}(X) = \PP^{N}$ - obviously $\sigma_{k}(X)$ is aCM.

\item \blue{Hypersurfaces}: $\sigma_{k}(X) \subset \PP^{N}$ is irreducible, so if $k$ is such that $\sigma_{k}(X) \subset \PP^{N}$ has codimension 1 then it is aCM.

\item \blue{Determinantal varieties} [Eagon, Eagon-Hochster]: If $X = \Seg(\PP A_{1}\times \PP A_{2})$,  then $\sigma_{k}(X)$ is a determinantal variety and aCM.

\item \blue{(Secretly) Determinantal varieties} \cite[Thm.~2.4]{CGG5_rational}: Suppose $X = \Seg(\PP A_{1}\times  \cdots \PP A_{d})$ and $Y = \Seg(\PP^{a}\times \PP^{b})$.  If $\sigma_{k}(X) = \sigma_{k}(Y)$, then $\sigma_{k}(X)$ is a determinantal variety and aCM.

\item \blue{Subspace varieties} \cite{Weyman}: $\Sub_{r_{1},\dots,r_{n}} = \{ T\in A_{1}\o \cdots \o A_{n} \mid \exists A_{i}'\subset A_{i}, \; \dim A_{i}' = r_{i},\; T \in A_{1}'\o \cdots \o A_{n}'   \}$.

\item \blue{GSS Conjecture for 4 factors} \cite{Landsberg-Weyman-Bull07}: $\sigma_{2}(\Seg(\PP^{n_{1}}\times \PP^{n_{2}} \times \PP^{n_{3}}\times \PP^{n_{4}}  ))$ is aCM for all $n_{i}\geq 1$.

\item \blue{Matrix Pencils}\cite{Landsberg-Weyman-Bull07}: $\sigma_{2}(\PP^{n_{1}}\times \PP^{n_{2}}  \times \PP^{n_{3}} )$ is aCM for all $n_{i} \geq 1$.
\item \blue{3rd secant of 3 factors} \cite{Landsberg-Weyman-Bull07}: $\sigma_{3}(\Seg(\PP^{n_{1}}\times \PP^{n_{2}} \times \PP^{n_{3}}  ))$  for all $n_{1}, n_{2},n_{3} \geq 1$.

\item \blue{A Strassen case}\cite{Landsberg-Weyman-Bull07}: $\sigma_{k}(\PP^{1}\times \PP^{n_{2}}  \times \PP^{n_{3}} )$ is aCM for all $k, n_{2},n_{3} \geq 1$.

\item \blue{A case with defective dimension}:  $\sigma_{3}(\Seg(\PP^{1}\times \PP^{1}\times \PP^{1}\times \PP^{1} ))$  is classically known to be a complete intersection (CI) of 2 quartics (choose any two of the three determinants of $4\times 4$ flattenings). It is aCM since CI implies aCM. It is one of the first non-matrix examples of a defective secant variety of a Segre product (it has codimension 2 and not 1 as expected).
\item \blue{New case} \cite{OedingSam}:  $\sigma_{5}(\Seg(\PP^{1})^{\times 5})$ is a complete intersection of two equations, one of degree 6 and one of degree 16, and is aCM since CI implies aCM.
\item \blue{Numerical results} \cite{DaleoHauenstein}: Using a numerical Hilbert function computation Daleo and Hauenstein were able to show that $\sigma_{4}(\PP^{2}\times \PP^{2}\times \PP^{3})$ is aCM with high probability. Later we will report a symbolic computation that removes the ``high probability'' qualifier.
\item \blue{Local results} \cite{MichalekOedingZwiernik}: $\sigma_{2}(\PP^{n_{1}}\times \dots \times \PP^{n_{m}})$ is covered by open normal toric varieties. In particular $\sigma_{2}(\PP^{n_{1}}\times \dots \times \PP^{n_{m}})$ is locally Cohen-Macaulay. 
\item\blue{Partially symmetric case}\cite[Lemma~5.5]{CEO}: For $r\leq 5$, $\sigma_{r}(\PP^{2}\times \nu_{2}\PP^{n})$ is defined by ``kappa-equations'' and moreover it is arithmetically Gorenstein (in particular it is aCM).
\end{itemize}

Since Veronese varieties are a close cousin to Segre varieties, we also collect the following evidence for secants of Veronese varieties. Geramita made the following conjecture which is a symmetric version of Conj.~\ref{conj:aCM}:
\begin{conjecture}[{\cite[p55]{Geramita_Lectures}}]
The $s$-secant variety of the degree  $d$ Veronese re-embedding of projective space $\sigma_{s}(\nu_{d}\PP^{n})$ is aCM for all $s,d,n$.
\end{conjecture}
As far as I know, the state of the art on this aCM question is  \cite[Thm. 1.56]{IarrobinoKanev_text}, which also gave the dimension, degree, and  singular locus.
\begin{theorem}[{\cite{Kanev_catalecticant}}] $\sigma_{s}(\nu_{d}\PP^{n})$ is aCM if either $d=2$, or $n=1$ or $s\leq 2$.
\end{theorem}

 Landsberg and Weyman applied a partial desingularization together with a mapping cone argument to show the following.

\begin{theorem} [{\cite{Landsberg-Weyman-Bull07}}]

Suppose $X:=\sigma_{r}(\Seg( {\PP^{r-1}}^{\times d}  ))$ is aCM, with ``a resolution by small partitions.'' If $n_{i}\geq r-1$ for all $1\leq i\leq d$, then
 $\sigma_{r}(\Seg( \PP^{n_{1}}\times \cdots \times \PP^{n_{d}} ))$ is aCM and its ideal is generated by those inherited from $X$ and the $(r+1)\times(r+1)$-minors of flattenings.
\end{theorem}

New cases found by Landsberg and Weyman using this result (we call it and its adaptations \emph{LW-lifting}) are the following:
\begin{itemize}
\item Direct computation: $\sigma_{2}(\Seg(\PP^{1} \times \PP^{1} \times \PP^{1} \times \PP^{1} ))$ is aCM with small partitions, and its ideal is defined by $3 \times 3$ minors of flattenings. 

\item LW-lifting implies that $\sigma_{2}(\Seg(\PP^{n_{1}}\times \PP^{n_{2}} \times \PP^{n_{3}}\times \PP^{n_{4}}  ))$ is aCM, and its ideal is defined by $3 \times 3$ minors of flattenings.

\item Direct computation shows that $\sigma_{3}(\Seg(\PP^{2}\times \PP^{2} \times \PP^{2}  ))$ is aCM with small partitions, and its ideal is defined by (Strassen's) 27 quartic equations. 

\item LW-lifting implies that 
$\sigma_{3}(\Seg(\PP^{n_{1}}\times \PP^{n_{2}} \times \PP^{n_{3}}  ))$ is aCM and ideal defined by quartic equations: those inherited from Strassen's and the $4\times 4$ minors of flattenings.
\end{itemize}

\section{Applications of LW-lifting and new examples of aCM secant varieties}\label{newACM}
Let $R = \CC[A\o B \o C]$ and $G = \GL(A)\times \GL(B) \times \GL(C)$.   
Galetto's \textsf{HighestWeights} package in Macaulay2,  \cite{HighestWeights} determines the $G$-module structure from the maps in a resolution of a $G$-invariant $R$-module (provided one can compute the resolution in the first place). Using this package, we can obtain the $G$-equivariant versions of all the resolutions we can compute, for example the case of $\sigma_{4}(\PP^{2}\times \PP^{2}\times \PP^{3})$ above. This symbolic tool should be useful for finding more examples of equivariant resolutions for secant varieties, which can inform further work on Conjecture \ref{conj:aCM}.
 
The Betti table and the $\GL(3)\times \GL(3)\times \GL(4)$-equivariant description (found using \textsf{HighestWeights} in Macaulay2) are as follows:
(We only record the Young tableau that index the $G$-modules in the resolution. The number of boxes in each factor captures the grading.)
\[
{\small
\ytableausetup{centertableaux,boxsize=0.5em}
      \hspace{1em}
\begin{array}{ccc}&\CC  \\
&\uparrow \\
 \ydiagram{2,2,2} \o \ydiagram{2,2,2} \o \ydiagram{3,1,1,1}
& \bigoplus &
\ydiagram{3,3,3} \o \ydiagram{3,3,3} \o \ydiagram{3,3,3} \\
&\uparrow \\
\ydiagram{4,3,3} \o \ydiagram{4,3,3} \o \ydiagram{3,3,2,2}
& \bigoplus &
\ydiagram{4,3,3} \o \ydiagram{4,3,3} \o \ydiagram{3,3,3,1}
\\
&\uparrow \\
\ydiagram{4,4,3} \o \ydiagram{4,4,3} \o \ydiagram{3,3,3,2}
& \bigoplus &
\begin{pmatrix}
\ydiagram{4,4,3} \o \ydiagram{5,3,3}\\  \bigoplus \\ \ydiagram{5,3,3} \o \ydiagram{4,4,3} 
\end{pmatrix}\o \ydiagram{3,3,3,2}
\\
&\uparrow \\
\ydiagram{4,4,4} \o\ydiagram{4,4,4} \o \ydiagram{3,3,3,3}
& \bigoplus &
\ydiagram{5,4,3} \o\ydiagram{5,4,3} \o \ydiagram{3,3,3,3}
\\
&\uparrow \\
&0
\end{array}}
\fbox{$\small \begin{matrix}
      &0&1&2&3&4\\\text{total:}&1&30&144&180&65\\\text{0:}&1&\text{.}&\text{.}&\text{.}&\text{.}\\\text{1:}&\text{.}&\text{.}&\text{.}&\text{.}&\text{.}\\\text{2:}&\text{.}&\text{.}&\text{.}&\text{.}&\text{.}\\\text{3:}&\text{.}&\text{.}&\text{.}&\text{.}&\text{.}\\\text{4:}&\text{.}&\text{.}&\text{.}&\text{.}&\text{.}\\\text{5:}&\text{.}&10&\text{.}&\text{.}&\text{.}\\\text{6:}&\text{.}&\text{.}&\text{.}&\text{.}&\text{.}\\\text{7:}&\text{.}&\text{.}&\text{.}&\text{.}&\text{.}\\\text{8:}&\text{.}&20&144&180&65\\\end{matrix}
$}
 \]

Note that $\sigma_{4}(\PP^{2}\times \PP^{2}\times \PP^{3})$ has codimension 4, and we have a length 4 resolution so it is aCM. This also confirms the Daleo-Hauenstein result unconditionally.
After checking that the ideal we started with is generically reduced, we obtain the following.
\begin{theorem}
The secant variety $\sigma_{4}(\Seg(\PP^{2}\times \PP^{2}\times \PP^{3}))$ is arithmetically Cohen-Macaulay.
Its prime ideal is minimally generated by the 10 degree 6 Landsberg-Manivel equations, and the 20 degree 9 equations inherited from Strassen's equation.
\end{theorem}
We note that the final modules in the resolution of $\sigma_{4}(\PP^{2}\times \PP^{2}\times\PP^{3})$ have ``small partitions,'' so we apply  our adaptation of LW-lifting (Thm.~\ref{thm:main}), to obtain:
\begin{theorem}
The secant variety $\sigma_{4}(\Seg(\PP^{2}\times \PP^{2}\times \PP^{n}))$ is arithmetically Cohen-Macaulay for all $n\geq 3$.
Its prime ideal is minimally generated by $5\times 5$ minors of flattenings, the degree 6 Landsberg-Manivel equations, and the degree 9 Strassen equations.
\end{theorem}

Other cases of where we can apply LW-lifting are the following:
\begin{prop}
\[\begin{array}{rcl}
\sigma_{4}(\PP^\red{2}\times \PP^\red{2}\times \PP^\red{2}) & \text{ is aCM, deg. 9 hypersurface } & \text{[Strassen]}. \\

\GL(4). \sigma_{4}(\PP^\red{2}\times \PP^\red{2}\times \PP^\red{2}) &\text{ is aCM  and $\codim$ 3 in}& \sigma_{4}(\PP^\red{2}\times \PP^\red{2}\times \PP^\blue{3}).
\\
\GL(4)^{\times 2}. \sigma_{4}(\PP^\red{2}\times \PP^\red{2}\times \PP^\red{2}) &\text{ is aCM and $\codim$ 4 in }&\sigma_{4}(\PP^\red{2}\times \PP^\blue{3}\times \PP^\blue{3}).
\\
\GL(4)^{\times 3}. \sigma_{4}(\PP^\red{2}\times \PP^\red{2}\times \PP^\red{2}) &\text{ is aCM and $\codim$ 5 in }& \sigma_{4}(\PP^\blue{3}\times \PP^\blue{3}\times \PP^\blue{3}).
\\
\GL(4). \sigma_{4}(\PP^\red{2}\times \PP^\red{2}\times \PP^\blue{3})&\text{  is aCM and $\codim$ 1 in }& \sigma_{4}(\PP^\red{2}\times \PP^\blue{3}\times \PP^\blue{3}).
\\
\GL(4)^{\times2} .\sigma_{4}(\PP^\red{2}\times \PP^\red{2}\times \PP^\blue{3})&\text{  is aCM and $\codim$ 2 in }&\sigma_{4}(\PP^\blue{3}\times \PP^\blue{3}\times \PP^\blue{3}).
\\
\GL(4). \sigma_{4}(\PP^\red{2}\times \PP^\blue{3}\times \PP^\blue{3})&\text{  has $\codim$ 1 in }&\sigma_{4}(\PP^\blue{3}\times \PP^\blue{3}\times \PP^\blue{3}).
\end{array}\]
\end{prop}

So what remains is to determine if we may lift the aCM property further.  If we can do this, this will complete a major step forward, since in particular it would solve the salmon conjecture \cite{prize, Friedland2010_salmon, Friedland-Gross2011_salmon, BatesOeding}.

We can also lift the result from \cite{OedingSam} as follows:
\begin{prop}
Let $G=\GL(n_{1})\times \dots \times \GL(n_{5})$, and let $X = \sigma_{5}((\PP^{1})^{\times 5})\subset \PP^{31}$. Then $G.X \subset \PP^{n_{1}n_{2}n_{3}n_{4}n_{5}-1}$ is aCM, and its ideal is generated by $3 \times 3$ minors of flattenings together with the $G$-module $\langle G.f_{6}\rangle \oplus \langle G.f_{16}\rangle$ where $f_{6}$ and $f_{16}$ are the minimal generators of the ideal of $X$. 
\end{prop}
\begin{proof}
Let $R = \CC[\CC^{2^{5}}]$.
The main result of \cite{OedingSam} is that $X$ is a complete intersection of $f_{6}$ and $f_{16}$, so it has resolution:
\[
0 \to R(-22) \to R(-6)\oplus R(-16) \to R 
.\]
The rank-1 module $R(-22)$ is a one-dimensional irreducible $GL(2)^{\times 5}$ module in degree 22, so it is described by the quintuple of partitions $(11,11)^{\times 5}$. Note $r = (2,2,2,2,2)$ and $\hat r_{j} - r_{j} = 16-2 = 14$ for all $j$. So the resolution is $\hat r_{j} - r_{j}$-small ($14>11$), and we can apply LW-lifting to obtain the result.
\end{proof}

\subsection{Matrix Pencils, $\sigma_{2}(\PP^{a}\times \PP^{b}  \times \PP^{c} )$ and $\sigma_{r}(\PP^{1}\times \PP^{b}\times \PP^{c})$}
Notice that $\sigma_{2}(\PP^{1}\times \PP^{1}\times \PP^{1})= \PP^{7}$, thus is aCM and trivially has a resolution by small partitions. Therefore by LW-lifting, $\sigma_{2}(\PP^{a}\times \PP^{b}  \times \PP^{c})$ is aCM for all $a,b,c \geq 1$.  If any of $a,b,c$ are equal to $0$, this reverts us back to the matrix case, and we know that the bounded-rank matrix varieties are all aCM. 
Thus
\begin{prop}[\cite{Landsberg-Weyman-Bull07}, \cite{CGG5_rational}] The following hold.

\begin{itemize}
\item $\sigma_{2}(\PP^{a}\times \PP^{b}  \times \PP^{c} )$ is aCM for all $a,b,c \geq 1$.

\item $\sigma_{k}(\PP^{1}\times \PP^{b}  \times \PP^{c} )$ is aCM for all $k, b,c \geq 1$.
\end{itemize}
\end{prop}
The second result follows from LW-lifting and a result of Strassen that says that $\Sub_{2,r,r} = \sigma_{r}(\PP^{1}\times \PP^{b} \times \PP^{c})$.   Some cases of this result are also contained in \cite[Thm.~2.4(ii,iii)]{CGG5_rational}. 
Here's a slight generalization.
\begin{prop}[{\cite[Thm.~7.3.1.3]{LandsbergTensorBook}}]
Suppose $X^{n}\subset \PP^{N}$ is a variety not contained in a hyperplane.  Then for all $s \geq N -n$,  $\sigma_{s}(X\times \PP^{r}) = \sigma_{s}(\PP^{N}\times \PP^{r})$ is defined by the $(s+1)\times (s+1)$ minors of a generic matrix, and thus is aCM.
\end{prop}
\begin{proof}
For the equality, see Landsberg's book. 
The fact about the ideal and the aCM properties follow from this isomorphism.
\end{proof}

\subsection{Arithmetically Gorenstein cases}\label{newAG}
Recall if $R$ is a polynomial ring and $I$ is an ideal, we say that the variety $\V(I)$ is \emph{arithmetically Gorenstein (aG)} if the coordinate ring $R/I$ is Gorenstein, which, in turn, can be detected by the following sufficient condition:
Suppose 
\[
I \leftarrow R^{b_{1}}\leftarrow \dots \leftarrow R^{b_{c}} \leftarrow 0
\]
is a free resolution of $I$ with length $c$ equal to the codimension of $I$ (i.e. $R/I$ is CM) and $b_{c}=1$.   This says that we can determine if the aG property holds if we know that the aCM property holds and the dual module of $I$ has rank 1. 

Fisher proved the following very nice result regarding secants of elliptic normal curves.
\begin{theorem}[\cite{Fisher}] Let $R = k[x_{1},\ldots,x_{n}]$ be the homogeneous coordinate ring of $\PP^{n-1}$. 
Let $C\subset \PP^{n-1}$ be an elliptic normal curve. If $m = n-2r \geq 2$, then $I(\sigma_{r}(C))$ has a minimal graded free resolution of the form
\begin{multline*}
0 \to R(-n) \to R(-n+r+1)^{b_{m-1}} \to R(-n+r+2)^{b_{m-2}} \\ \to \dots \to R(-r-2)^{b_{2}} \to R(-r-1)^{b_{1}} \to R \to 0
.\end{multline*}
In particular, $\sigma_{r}(C)$ is projectively Gorenstein of codimension $m$.
\end{theorem}
The following is well known and follows, for instance, from Lascoux's theorem.
\begin{prop}
For all $r< \min\{a,b\}$, the matrix secant variety $\sigma_{r}(\PP^{a} \times \PP^{b})$ is arithmetically Gorenstein if and only if $a=b$.
\end{prop}
From this result we immediately have 
\begin{prop} The following varieties are aG:
\[
\sigma_{r}(\PP^{1}\times \PP^{r-1}\times \PP^{2r-1}) \cong \Sub_{2,r,r} \CC^{2} \o \CC^{r} \o \CC^{2r} \cong \Sub_{(2r),r} (\CC^{2} \o \CC^{r}) \o \CC^{2r}
\]
\end{prop}
\begin{remark} We leave it as an open question to use LW-lifting to prove precisely which other cases of $\sigma_{k}(\PP^{1}\times \PP^{a}\times \PP^{b})$ are aG.
\end{remark}

Now we focus on the case $r=2$.
In \cite{MichalekOedingZwiernik} Michalek, Zwiernik, and the author found all cases where $\sigma_{2}(\PPA)$ is locally Gorenstein:
\begin{theorem}[\cite{MichalekOedingZwiernik}]
The only cases where $\sigma_{2}(\PPA)$  is (locally) Gorenstein are
\begin{center}
$\sigma_2(\PP ^a\times\PP ^a)$, \quad  $\sigma_2(\PP ^1\times\PP ^k) = \PP^{2k+1}$,

$ \sigma_{2}(\PP^{1}\times \PP^{1}\times \PP^{1})=\PP^{7}$,\quad
$\sigma_2(\PP ^1\times\PP ^1\times\PP ^3)$, \quad
$\sigma_2(\PP ^1\times\PP ^3 \times\PP ^3)$, \quad
$\sigma_2(\PP ^3 \times \PP ^3 \times\PP ^3)$, 

$\sigma_2(\PP ^1\times\PP ^1\times\PP ^1\times\PP ^1\times\PP ^1)$.
\end{center}
\end{theorem}
\begin{prop} The following varieties are aG:
\[
\sigma_{k}(\PP^{1}\times \PP^{k-1}\times \PP^{2k-1}) \cong \Sub_{2,k,k} \CC^{2} \o \CC^{k} \o \CC^{2k} \cong \Sub_{(2k),k} (\CC^{2} \o \CC^{k}) \o \CC^{2k}
\]
\end{prop}
\begin{remark} We leave it as an open question to use LW-lifting to prove precisely which other cases of $\sigma_{k}(\PP^{1}\times \PP^{a}\times \PP^{b})$ are aG.
\end{remark}

\begin{table}
\begin{tabular}{|c|c|c|}
\hline
variety & codimension & socle degree \\
\hline
$\sigma_2(\PP ^{n-1}\times\PP ^{n-1})$ & $(n-2)^{2}$ &$2n-3$
\\ \hline
$\sigma_2(\PP ^1\times\PP ^1\times\PP ^3)$& $4$ & $5$
\\ \hline
$\sigma_2(\PP ^1\times\PP ^3 \times\PP ^3)$& $16$ & $9$
\\ \hline
$\sigma_2(\PP ^3 \times \PP ^3 \times\PP ^3)$& $44 $ & $13$
\\ \hline
$\sigma_2(\PP ^1\times\PP ^1\times\PP ^1\times\PP ^1\times\PP ^1)$ & $20 $ &  ?
\\ \hline
\end{tabular}
\caption{The non-trivial arithmetically Gorenstein secant varieties and their socle degrees. (The last row conjecturally aG.)}\label{tab:gor}
\end{table}

\begin{theorem}
$\sigma_{2}(\PPA)$  is arithmetically Gorenstein in the following cases
\begin{center}
$\sigma_2(\PP ^a\times\PP ^a)$, \quad  $\sigma_2(\PP ^1\times\PP ^k) = \PP^{2k+1}$,

$ \sigma_{2}(\PP^{1}\times \PP^{1}\times \PP^{1})=\PP^{7}$,\quad
$\sigma_2(\PP ^1\times\PP ^1\times\PP ^3)$, \quad
$\sigma_2(\PP ^1\times\PP ^3 \times\PP ^3)$, \quad
$\sigma_2(\PP ^3 \times \PP ^3 \times\PP ^3)$, 
\end{center}
The socle degrees are listed in table \ref{tab:gor}.
\end{theorem}

\begin{remark}
We also conjecture that $\sigma_{2}((\PP^{1})^{\times 5})$ is arithmetically Gorenstein, but we are not able to prove this at present.
\end{remark}

\begin{proof}
Note the matrix cases, $\sigma_2(\PP ^{n-1}\times\PP^{n-1})$,   $\sigma_2(\PP ^1\times\PP ^k) = \PP^{2k+1}$  and   $\sigma_{2}(\PP^{3}\times \PP^{3}) \cong \sigma_{2}(\PP^{1}\times \PP^{1}\times \PP^{3})$  being square are all arithmetically Gorenstein (see for instance \cite{Weyman}[Cor.~6.1.5(c)]).
The socle degree for the ideal $I_{r}$ of $(r+1)\times(r+1)$-minors of an $n\times n$ matrix can be computed in a round-about way as follows. As noted in the proof of Cor.~6.1.5(c)] \cite{Weyman}, the 
last module in the resolution is
\[
S_{(n)^{(n-r)}}A \o S_{(n)^{(n-r)}}B
,\]
which is indexed by square partitions with $n(n-r)$ boxes. In all of the cases we consider, the number of boxes in the last module is equal to one less than the length plus the width of the Betti table of the resolution. Since $I_{r}$ is CM of codimension $(n-r)^{2}$, the length is $(n-r)^{2}$, so the regularity (socle degree) must be $r(n-r)+1$. In the case of rank 2 matrices, the regularity is $2(n-2)+1 = 2n-3$.

Now  consider the 3-factor cases.
We note that $\sigma_{2}(\PP^{1}\times \PP^{1}\times \PP^{3})$ is isomorphic to
\[
\Sub_{2,2,2}(\CC^{2}\o \CC^{2}\o \CC^{4})
.\]
which is, in turn, isomorphic to the variety of $4\times 4$ matrices of rank 2. This ideal is arithmetically Gorenstein since the ideal of $3 \times 3$ minors of a generic $4\times 4$ matrix is generated by the Jacobian ideal of the $4\times 4$ determinant. The resolution is given by the Lascoux resolution whose form and Betti table are:
\[
R(-4) \leftarrow R^{16}(-7) \leftarrow R^{30}(-8) \leftarrow R^{16}(-9) \leftarrow R(-12) \leftarrow 0
\quad
\fbox{$
\begin{matrix}
      &0&1&2&3&4\\
      \text{total:}&1&16&30&16&1\\
      \text{0:}&1&\text{.}&\text{.}&\text{.}&\text{.}\\
      \text{1:}&\text{.}&\text{.}&\text{.}&\text{.}&\text{.}\\
      \text{2:}&\text{.}&16&30&16&\text{.}\\
      \text{3:}&\text{.}&\text{.}&\text{.}&\text{.}&\text{.}\\
      \text{4:}&\text{.}&\text{.}&\text{.}&\text{.}&1\\
      \end{matrix}
$}
\]

Let $A_{1}\cong A_{2} \cong \CC^{2}$, $A_{3}\cong \CC^{4}$. The equivariant version of the Lascoux resolution is (we omit the twists)

\begin{multline*}
\mathbb C 
\leftarrow \bw{3}(A_{1}\o A_{2})\o \bw{3}(A_{3}) \\
\leftarrow \begin{matrix} \bw{4}(A_{1}\o A_{2}) \o S_{2,1,1}A_{3}  \\ \oplus \\S_{2,1,1}(A_{1}\o A_{2}) \o \bw{4} A_{3}\end{matrix}
\leftarrow S_{2,1,1,1}(A_{1}\o A_{2}) \o S_{2,1,1,1}A_{3} 
\\
\leftarrow S_{2,2,2,2}(A_{1}\o A_{2}) \o S_{2,2,2,2}A_{3} 
\leftarrow 0
\end{multline*}
The last module of the resolution is isomorphic to $S_{4,4}A_{1}\o S_{4,4}A_{2}\o S_{2,2,2,2}A_{3}$, which is $\widehat{r}_{j}-r_{j}$ small for $r = (2,2,4)$ and $a=(2,4,4)$, so we can apply the LW-lifting method and lift this resolution to a resolution for $\sigma_{2}(\PP^{1}\times \PP^{3}\times \PP^{3})$. If we go directly from $\sigma_{2}(\PP^{1}\times \PP^{1}\times \PP^{3})$ to $\sigma_{2}(\PP^{1}\times \PP^{3}\times \PP^{3})$ the relevant quotient bundle $\Q_{B}$ would have rank 2, and the bundle $\bw{i+j} \xi$ will no longer be irreducible in the $B$-factor. In order to avoid passing the associated graded bundle, computing cohomology of the associated graded, and then using spectral sequences to reconstruct the cohomology of the original bundle, we prefer to lift one dimension at a time, keeping all quotient bundles rank 1 and keeping all $\bw{d}\xi$ bundles irreducible over the relevant base. 

We will lift the resolution of $\sigma_{2}(\PP^{1} \times \PP^{1}\times \PP^{3})$ to $\sigma_{2}(\PP^{1} \times \PP^{2}\times \PP^{3})$ and then $\sigma_{2}(\PP^{1} \times \PP^{3}\times \PP^{3})$, and further lift to $\sigma_{2}(\PP^{2} \times \PP^{3}\times \PP^{3})$ and then to $\sigma_{2}(\PP^{3} \times \PP^{3}\times \PP^{3})$.

Let $A_{2}' \cong \CC^{3}$ and $A_{2}'' \cong \CC^{4}$.  We wrote a script in LiE to compute the relevant cohomology for all of the following results.
  The strand in the mapping cone construction above $S_{4,4}A_{1}\o S_{4,4}\R_{2}'\o S_{2,2,2,2}A_{3}$ when lifting $A_{2}'$ ends in the module $S_{8,8}A_{1}\o S_{6,5,5}A_{2}' \o S_{4,4,4,4}A_{3}$. Since this is the corner of the square in the mapping cone which will produce a resolution for $\sigma_{2}(\PP^{1} \times \PP^{2}\times \PP^{3})$ we know that this is the last $R$-module in that resolution. Since this module is not 1-dimensional this gives another indication that $\sigma_{2}(\PP^{1} \times \PP^{2}\times \PP^{3})$ is not aG. 

By the same method we lift $S_{8,8}A_{1}\o S_{6,5,5}R_{2}'' \o S_{4,4,4,4}A_{3}$ to produce the last module in the resolution of  $\sigma_{2}(\PP^{1} \times \PP^{3}\times \PP^{3})$, which is the 1-dimensional module $S_{12,12}A_{1}\o S_{6,6,6,6}A_{2}''\o S_{6,6,6,6}A_{3}$. This proves that  $\sigma_{2}(\PP^{1} \times \PP^{3}\times \PP^{3})$ is indeed aG. Since this module is indexed by partitions with $24$ boxes, and the codimension is $16$ the regularity must be $9$.

Now we do the same procedure in the first factor, lifting the resolution of  $\sigma_{2}(\PP^{1} \times \PP^{3}\times \PP^{3})$ to a resolution of   $\sigma_{2}(\PP^{2} \times \PP^{3}\times \PP^{3})$ and then to a resolution of  $\sigma_{2}(\PP^{3} \times \PP^{3}\times \PP^{3})$. Let $A_{1}' \cong \CC^{3}$ and $A_{1}'' \cong \CC^{4}$. 
We find that $S_{12,12}R_{1}\o S_{6,6,6,6}A_{2}''\o S_{6,6,6,6}A_{3}$ lifts to the 3-dimensional module $S_{14,13,13}A_{1}'\o S_{10,10,10,10}A_{2}''\o S_{10,10,10,10}A_{3}$. And then
$S_{14,13,13}R_{1}'\o S_{10,10,10,10}A_{2}''\o S_{10,10,10,10}A_{3}$ lifts to 
the 1-dimensional module
$S_{14,14,14,14}A_{1}''\o S_{14,14,14,14}A_{2}''\o S_{14,14,14,14}A_{3}$. This proves that $\sigma_{2}(\PP^{3}\times \PP^{3}\times \PP^{3})$ is aG.
Since this module is indexed by partitions with $56$ boxes, and the codimension is $44$ the regularity must be $13$.
\end{proof}

\begin{appendix}
\part*{Appendix: Tensors, Weyman's geometric technique, and MCM modules}
\section{Orbit closures in tensor spaces}\label{sec:orbits}
Our indexing convention is as if a vector space is defined without indices, we get an analogous object adding indices.  An underlined vector space such as $\u{A}$ will indicate the trivial vector bundle with fiber $A$.
Define $A^{*}$ to be a vector space (over $\CC$) of dimension $a$.\footnote{We prefer to work with the dual vector space so that the functions on these spaces will not have the $^{*}$'s.} 
Then for each $j$, $A_{j}^{*}$ is defined as a vector space of dimension $a_{j}$. Let $[n]$ denote the multi-index $\{1,\dots,n\}$, and let $I =[i_{1},\dots,i_{s}] \subset [n]$. Then we have compact notations $A_{[n]}^{*}:=\AA$ and $A_{I}^{*}:=A_{i_{1}}^{*}\o\dots\o A_{i_{s}}^{*}$ and these extend to Schur modules as $(S_{\pi}A)_{[n]}:= S_{\pi^{n}}A_{1}\o \dots \o S_{\pi^{n}}A_{n}$.

If a group $G$ acts on $\PP^{N}$, an algebraic variety $Y\subset \PP^{N}$ is said to be a \emph{$G$-variety} if $G.Y = Y$.  A particular type of $G$-variety is an \emph{orbit closure}, which is a variety $\overline{G.x}$ for some $[x]\in \PP^{N}$. Such a point $x$ is called a \emph{normal form} for $X$.\footnote{Notice a point  $[x]\in \PP^{N}$ can be considered the $\CC^{*}$ orbit closure of a non-zero point $x$ in affine space. By considering the $GL$-action instead of the $SL$-action our normal forms will be points in affine space.}
Herein we are most interested in $G$-varieties and orbit closures for the group $\GLA$ acting naturally on $\AA$ by change of coordinates in each factor.

Classically there has been much interest in classification problems related to counting orbits. For example, Kac \cite{Kac80, Kac85} determined which pairs of vector spaces and group actions produce finitely many orbits, and Vinberg and coauthors produced lists of orbits in many cases (see \cite{Vinberg-Elasvili, Nurmiev, Nurmiev2} for example).  More recently, Buczy{\'n}ski and Landsberg found expressions of normal forms for all orbits in the third secant variety  \cite{BucLan_ranks}.  These results for tensors are collected in \cite[Ch.~10]{LandsbergTensorBook}.

\subsection{Classical varieties as orbit closures}
Indeed, many interesting algebraic varieties can be described as orbit closures.

For example, the Segre variety, denoted $\Seg(\PPA)$, of (lines through) rank-one tensors in
 $\PP (\AA)$ is the (already closed) orbit $\GLA$ of a point $[\xx]$, where $0\neq x_{i}\in A_{i}^{*}$.  More generally, closed orbits are called \emph{homogeneous varieties} and include other classical varieties such as the Grassmannian,
flag varieties,
Veronese varieties
and 
Segre-Veronese varieties.
In general, if $Y \subset \PP V$ is a variety, $\w{Y}$ will denote the cone over $Y$ in $V$.

Tangential varieties of Segre products, denoted $\tau(\PPA)$ are orbit closures:
\[
\w{\tau}(\PPA) = \overline{
\GLA. \begin{pmatrix} 
\phantom{+} 
 y_{1}\o x_{2} \o\dots\o x_{n}
\\ 
+ x_{1}\o y_{2}\o x_{3}\dots\o x_{n}
\\
\vdots
\\
+ x_{1}\o\dots\o x_{n-1}\o y_{n}
 \end{pmatrix}}
,\]
with $\texttt{span}\{x_{i},y_{i}\}$ maximal in  $A_{i}^{*}$.

Many secant varieties of Segre products,  $\sigma_{k}(\PPA)$, are orbit closures. For example, if $a_{i}\geq k$, then 
$\sigma_{k}(\PPA)$ is the orbit closure of the tensor with 1's on the super-diagonal and 0's elsewhere. 
When the dimensions $a_{i}$ are smaller than $k$ it is not immediately obvious how to find a normal form for the secant variety $\sigma_{k}(\PPA)$ or even when the secant variety is still the closure of a single orbit.  Indeed, because for large enough $k$ the secant variety fills the ambient space, if the dimension of the group is smaller than the dimension of the ambient space, then an orbit closure cannot be the entire space.

On the other hand,
the affine secant variety $\w{\sigma_{k}}(\PPA)$ may be parametrized as (the closure of) the composition of the Segre map with the summation map:
\[\xymatrix{
(\tA)^{\times k} \ar[r]^{Seg}& (\AA)^{\times k} \ar[r]^{sum}& \AA
},\]
which shows in particular that the variety is irreducible, also, we expect that the dimension of the affine secant variety will then be $\min\{k\sum_{j}a_{j}, \prod_{j}a_{j} \}$.
So if we allow a normal form to depend on enough parameters, we can always parameterize a secant variety by the closure of the orbit of such a parametrized normal form.

\subsection{Stability for normal forms}
A fundamental family of varieties in $\PP (\AA)$ are subspace varieties.
The \emph{subspace variety}, $\SubA$ is the (affine) variety of all points $T \in \AA$ such that there exist subspaces $A_{j}'^{*}\subset A_{j}^{*}$ of dimensions $r_{j}$ respectively and $T\in \AAp$.

It is well-known \cite{Landsberg-Weyman-Bull07} that the ideals of the subspace varieties are generated by minors of flattenings, whose definition we now recall.
For a multi-index $I = (i_{1},\dots,i_{s})$ let $A_{I}$ denote the tensor product $\bigotimes_{j=1}^{s}A_{i_{j}}$. 
If $T \in A_{1}^{*}\o\dots\o A_{n^{*}}$, a \emph{flattening} is a linear map induced from $T$ of the form $A_{J}\to A_{[n] \setminus J}^{*}$.
Landsberg and Weyman \cite[Thm~3.1]{Landsberg-Weyman-Bull07} proved that the ideal of the $\SubA$ is generated by all $(r_{i}+1)\times (r_{i}+1)$ minors of the flattenings $A_{i}\to A_{[n]\setminus \{i\}}^{*}$.
As modules Landsberg and Weyman's result says that
\[\I(\SubA )= \left\langle 
\sum_{j=1}^{n} \bw{r_{j}+1} A_{j}
\o \bw{r_{j}+1}(A_{[n]-j})
\right\rangle.
\]

Now we have the language to discuss stability of orbit closures. 
For each $j$ choose a subspace   $A_{j}'^{*}\subset A_{j}^{*}$ of dimension $r_{j}$. Suppose we have a normal form $x\in \AAp$ for $\sigma_{k}(\PPAp) \subset \PP (\AA)$ (with the $\GLAp$-action). If  $x$ is also a normal form for $\sigma_{k}(\PPA)$ (with the $\GLA$-action)  we say that the normal form \emph{stabilizes}.
If the secant variety $\sigma_{k}(\PPAp)$
is not all of $\PP (\AAp)$ then the normal form stabilizes if and only if $r_{j}\geq k $ for every $j$.

\begin{example}
Consider the family of third secant varieties together with normal forms in Table \ref{tab:sig3}, which is adapted from  \cite[Ch.~10]{LandsbergTensorBook}. We assume that the $x_{i}$ (respectively $y_{i}$ and $z_{i}$) are linearly independent.
\begin{table}
\[
\begin{array}{|c|c|c|}
\hline
\text{orbit closure} &\text{normal form} & \text{group acting}
\\
\hline
\sigma_{3}(\PP^{a_{1}-1}\times \PP^{a_{2}-1}\times \PP^{a_{3}-1})
& x_{1}\o y_{1}\o z_{1} + x_{2}\o y_{2}\o z_{2}+ x_{3}\o y_{3}\o z_{3}
& \GL(a_{1})\times \GL(a_{2})\times \GL(a_{3})
\\
\hline
\sigma_{3}(\PP^{1}\times \PP^{a_{2}-1}\times \PP^{a_{3}-1})
& x_{1}\o (y_{1}\o z_{1} + y_{2}\o z_{2}) +  x_{2}\o (y_{2}\o z_{2} + y_{3}\o z_{3}) 
& \GL(2)\times \GL(a_{2})\times \GL(a_{3})
\\
\hline
\sigma_{3}(\PP^{1}\times \PP^{1}\times \PP^{a_{3}-1})
& x_{1}\o (y_{1}\o z_{1} + y_{2}\o z_{2}) +  x_{2}\o (y_{1}\o z_{2} + y_{2}\o z_{3}) 
& \GL(2)\times \GL(2)\times \GL(a_{3})
\\
\hline
\end{array}
\]
\caption{Some orbit closures for the third secant variety for $a_{1},a_{2},a_{3}\geq 3$ }\label{tab:sig3}
\end{table}

These three varieties do not have the same normal form.
In particular, consider each normal form in $A_{1}^{*}\o A_{2}^{*} \o A_{3}^{*}$ and let $\GL(a_{1})\times \GL(a_{2})\times \GL(a_{3})$ act, there will be a strict containment of varieties
\begin{multline*}
(\GL(a_{1})\times \GL(a_{2})).
\sigma_{3}(\PP^{1}\times \PP^{1}\times \PP^{a_{3}-1})
\\
\subset
\GL(a_{1}).
\sigma_{3}(\PP^{1}\times \PP^{a_{2}-1}\times \PP^{a_{3}-1})
\\
\subset
\sigma_{3}(\PP^{a_{1}-1}\times \PP^{a_{2}-1}\times \PP^{a_{3}-1})
.\end{multline*}
\end{example}
Subspace varieties aid the study of geometric and algebraic properties of these types of orbit closures.

\subsection{Lifting orbits}

The reference for this Section is \cite{Weyman, Landsberg-Weyman-Bull07}.  We emphasize, we are repeating what is in \cite[Section~5]{Landsberg-Weyman-Bull07}.

Suppose $A$ is a vector space of dimension $a$, and consider the Grassmannian
$\Gr(r,A^{*})$. We have the tautological sequence of vector bundles on $\Gr(r,A^{*})$:
 \[\xymatrix{
 0\ar[r] &\R \ar[r] & \u A \ar[r] & \Q \ar[r]& 0
 }
, \]
where $\R$ and $\Q$ are the subspace and quotient bundles (of ranks $r$ and $a-r$ respectively) and $\u A$ is the trivial vector bundle with every fiber isomorphic to $A$.

Subspace varieties have nice desingularizations:
\begin{equation}\label{eq:desing}
\xymatrix{
\R_{1}\o \R_{2}\o \dots\o \R_{n} 
\ar[d] \ar[dr]\\
\Gran& \Sub_{r_{1},\dots,r_{n}}(\AA)
.}
\end{equation}

In particular, we see immediately that
\[
\dim(\SubA) = r_{1}r_{2}\cdots r_{n} + \sum_{j=1}^{n}r_{j}(a_{j}-r_{j})
.\]

Let $G:=\GLA$, 
suppose we have an algebraic variety $Y\subset \PP( \AAp)$ and consider the orbit closure $\overline{G.Y}$.
The desingularization \eqref{eq:desing} can be used to give a partial desingularization of $\overline{G.Y}$.
We are most often interested in the case when $Y$ is an orbit closure or even just invariant under the action of $\GLAp$ for subspaces $A'^{*}_{j}\subset A_{j}^{*}$.

Again consider the bundle 
$\R_{1}\o \dots \o \R_{n}$, with total space $\tilde Z$. Each fiber over $(A'^{*}_{1},\dots,A'^{*}_{n}) \in \Gran$ is $\AAp$. 


Let $Z \subset \tilde Z$ be the subvariety such that the fibers $Z_{A'^{*}_{1},\dots,A'^{*}_{n}} \cong Y$.
Then $Z$ gives a partial desingularization of the big orbit closure $\overline{G.\w{Y}}$:
\[
\xymatrix{
\R_{1}\o \R_{2}\o \dots\o \R_{n} 
\ar[d] \ar[dr]\\
\Gran & \overline{G.\w{Y}} &\subset& \SubA
.}
\]
\begin{remark}
Landsberg and Weyman used this construction with $Y =  \sigma_{r}(\PP^{r-1}\times \dots\times \PP^{r-1})$.
\end{remark}
Now $\overline{G.Y}$ is birational to $\Gran \times Y$  so the dimension of $\overline{G.Y}$ is 
\begin{equation}
\dim(\overline{G.Y}) =  
\sum_{j=1}^{n}r_{j}(a_{j}-r_{j})
+
\dim(Y).
\end{equation}
We highlight the following fact, which is a straightforward generalization of \cite[Prop.~5.1]{Landsberg-Weyman-Bull07}, and note that Landsberg and Weyman's proof still holds in this generality:
\begin{prop}
The ideal of $\overline{G.Y}$ is generated by the equations of $\SubA$, i.e. the appropriate minors of flattenings together with the equations inherited from $Y$. 
\end{prop}
Recall that the process of ``inheritance'' is taking every module $S_{\pi}A'$ in a generating set for $\I(Y)$ and replacing it with $S_{\pi}A$ to get a module in $\I(\overline{G.Y})$.

In the next section we will recall how Weyman's ``geometric technique'' exploits this partial desingularization to get information about the minimal free resolution of $\overline{G.Y}$ from the minimal free resolution of $Y$.

\section{Obtaining resolutions via Weyman's Technique}\label{sec:weyman}
Landsberg and Weyman showed that under a minor technical hypothesis, the aCM property is inherited, from $\sigma_{r}(\PP^{r-1}\times\dots\times \PP^{r-1})$
to $\sigma_{r}(\PP A^{*}_{1}\times \dots \times \PP A_{n}^{*})$ when $\dim A_{j} \geq r$  for $1 \leq j \leq n$, \cite[Lem.~5.3]{Landsberg-Weyman-Bull07}.
 But their result does not say anything about the cases when the dimensions of $A_{j}$ are  \emph{smaller} than $r$. 
 It turns out that their same argument can be applied to any normal orbit closure that is partially resolved by a subspace variety. 

The goal of the next several sections is to prove a
 generalization of \cite[Lem.~5.3]{Landsberg-Weyman-Bull07}, which is Theorem~\ref{thm:workhorse} below.

Here we recall some definitions from \cite{Eisenbud_tome}.
Let $S$ denote the coordinate ring $\CC[A]$. Recall a \emph{minimal free resolution} of $\CC[X]$ is a complex of $S$-modules:
\[
F_{0}\leftarrow F_{1}\leftarrow\dots \leftarrow F_{p}\leftarrow 0
\]
such that $\CC[X]$ is the cokernel of the map $F_{0}\leftarrow F_{1}$.
We say that $X$ is \emph{arithmetically Cohen-Macaulay} (aCM) if the length $p$ of a minimal free resolution of $\CC[X]$ is equal to the codimension of $X$. 
 
The homogeneous coordinate ring $\CC[\AAn]$ has an isotypic decomposition by $G$-modules \cite[Prop.~4.1]{Landsberg-Manivel04} 
\[
\CC[\AAn] = \bigoplus_{d} \bigoplus_{\pi\vdash d} S_{\pi}A \o \CC^{m_{\pi}}
,\]
where $S_{\pi}A:=S_{\pi^{1}}A_{1}\o\dots\o S_{\pi^{n}}A_{n}$ denote  Schur modules indexed by multi-partitions $\pi = (\pi^{1},\dots,\pi^{n})$ with each 
$\pi^{j}=(\pi^{j}_{1},\dots,\pi^{j}_{a_{j}})$ a non-increasing sequence of integers and  $\CC^{m_{\pi}}$ denotes the multiplicity space of dimension equal to the multiplicity $m_{\pi}$ of the representation $S_{\pi}A\subset S^{d}(A)$.  Thus every $G$-equivariant resolution also has an isotypic decomposition.

We will say that a variety $Y\subset\PP( \AA)$ is a $G$-variety for $G=\GLA$ if $Y$ is left invariant under the action of $G$. We will consider $G$-equivariant free resolutions whose free $\CC[\AAn]$-modules also carry the structure of $G$-modules and are indexed by multi-partitions. 

We will use the following version of Weyman's Theorem~\cite{Weyman}, combined with Prop 2.2 of \cite{LanWey_secant_CHSS}:

\begin{theorem}[{\cite[Thm.~5.1.2]{Weyman},\cite[Prop.~2.2]{LanWey_secant_CHSS}}]
Suppose $Y \subset \PP V$ is a $G$-variety and consider the homogeneous variety $\B = G/P$. Suppose $q\colon \u{E}\to \B$ is a sub bundle of the trivial bundle $\u{V}$ and that $q\colon \PP E \to Y$ is a desingularization of $Y$.  Write $\eta = \u{E}^{*}$ and $\xi = \u{V} / \u{E}$. 

Suppose $\eta$ is  induced from an irreducible $P$-module. Then
\begin{itemize}
\item $\w{Y} $ is normal with rational singularities.
\item The coordinate ring $\CC[\w{Y}] \simeq H^{0}(\B, S^{d}\eta)$.
\item The vector space of minimal generators of the ideal of $Y$ in degree d is isomorphic to $H^{d-1}(\B, \bw{d}\xi)$.
\item More generally, $\bigoplus_{j} H^{j}(\B, \bw{i+j}\xi)$ is isomorphic (as $G$-modules) to the $i$-th term in the minimal free resolution of $Y$.
\end{itemize}
\end{theorem}

We will say that a $G$-variety $Y$ has an \emph{$(s_{j})$-small resolution} if for every module $S_{\pi}A$ occurring in the resolution has the property that for each $j$ the first part of $\pi^{j}$ is not greater than $s_{j}$.  Let $\w{a_{j}}:=\frac{a_{1}\dots a_{n}}{a_{j}}$ and similarly $\w{r_{j}} :=\frac{r_{1}\dots r_{n}}{r_{j}}$.

\begin{theorem}\label{thm:workhorse} Suppose $G = \GLA$ and $G' = \GLAp$.
If $Y$ is a $G'$-variety that is arithmetically Cohen-Macaulay with a resolution 
 that is $(\w{r_{j}} -r_{j})$-small 
 for every $j$ for which $0<r_{j} <a_{j}$,
then $\overline{G.Y}$ is arithmetically Cohen-Macaulay. 
Moreover 
we obtain a (not necessarily minimal) resolution of $\overline{G.Y}$ that is $(s_{j})$-small with
$s_{j} = 
\max_{\pi}
\begin{cases}
 \w{a_{j}}-r_{j}, & \text{ if } r_{j}<a_{j} \\
 \w{a_{j}} - \w{r_{j}}+\pi^{j}_{1}, & \text{ if }r_{j}=a_{j},
\end{cases}$
where the $\max$ is taken over all multi-partitions $\pi$ occurring in the original resolution of $\CC[Y]$.
\end{theorem}
The ``moreover'' statement in Theorem~\ref{thm:workhorse} did not appear in Landsberg and Weyman's version.  It allows the theorem to be used iteratively, as we will do in Section~\ref{}.  

\section{Bott's algorithm and Cohomology}\label{sec:Bott}
In this section we perform standard calculations (we use the Borel-Weil-Bott theorem throughout) that we will need to prove technical lemmas that will be used to prove Theorem~\ref{thm:workhorse}.

We will consider vector bundles (associated to irreducible $P$-modules) 
$S_{\pi}\R^{*}$ over the Grassmannian $\Gr(r,A^{*}) = G/P$.  We say that a bundle $\M$ is \emph{acyclic} if it has no higher cohomology.

Suppose $\pi \in \ZZ^{r}$ and $\lambda\in \ZZ^{a-r}$.
To the bundle $S_{\pi}R^{*}\o S_{\lambda}\Q$ over the Grassmannian $\Gr(r,A)$ we associate the weight
\[
w(\pi | \lambda) = [\pi_{1}-\pi_{2},\pi_{2}-\pi_{3},\dots,\pi_{r-1}-\pi_{r},\pi_{r} -\lambda_{1},\lambda_{1}-\lambda_{2}, \dots,\lambda_{a-r}-\lambda_{a-r-1}] \in \ZZ^{a-1}.
\] 
If $\pi =(\pi_{1},\dots,\pi_{r}) \in \ZZ^{r}$ let $\overleftarrow{\pi}$ denote the reverse $\overleftarrow{\pi} = (\pi_{r},\dots, \pi_{1})$.    When it is clear that we are considering partitions, we let $(l^{r})$ denote the partition $(l,\dots,l) \in \ZZ^{r}$.  So, we have a natural isomorphism with the sheaf corresponding to the \emph{contragradient} representation
\begin{equation}\label{eq:dualdual}
(S_{\pi}\R^{*} )^{*} \cong S_{- \overleftarrow{\pi}} \R^{*},
\end{equation}
which is sometimes convenient to write as
\begin{equation}\label{eq:dual}
S_{\pi}\R^{*} \o (\bw{r}\R)^{l} \cong S_{l^{r} - \overleftarrow{\pi}} \R
\end{equation}

We have similar definitions for $\Q$, so that naturally, 
\begin{equation}
S_{\lambda}Q^{*}\o (\bw{a-r}Q)^{m} \cong S_{m^{a-r}-\overleftarrow{\lambda}}\Q
\end{equation}
 as sheaves of $\GL(a-r)$-modules.

A weight $w = [w_{1},\dots,w_{a-1}]$ is said to be \emph{dominant} if $w\in \ZZ_{\geq0}^{a-1}$. 

If $\pi = (\pi_{1},\dots,\pi_{r})$ is a decreasing sequence of non-negative integers, the weight $w(\pi)$ is dominant and the Borel-Weil theorem says
 \[H^{0}(\Gr(r,A^{*}), S_{\pi}R^{*}) = S_{\pi}A.\]
 
The Weyl group $\W:=\mathcal{S}_{a-1}$ acts on weights as follows. Let $s_{1}\dots s_{a-1}$ denote the simple reflections. Simple reflections act locally on weights:

\[
s_{1}([w_{1},\dots,w_{a-1}]) = [-w_{1},w_{1}+w_{2},w_{3},\dots,w_{a-1}]
,\]
\[
s_{i}([w_{1},\dots,w_{a-1}]) = [w_{1}\dots,w_{i-2},w_{i-1}+w_{i},-w_{i},w_{i}+w_{i+1},w_{i+2},\dots,w_{a-1}], \quad \text{for } 2\leq i\leq a-2,
\]
\[
s_{a-1}([w_{1},\dots,w_{a-1}]) = [w_{1}\dots,w_{a-3},w_{a-2}+w_{a-1},-w_{a-1}]
.\]

Let $\rho = [1,1,\dots,1]$.
The affine action on weights is $s.w = s(w+\rho)-\rho$.
  We say that a weight $w$ is \emph{singular} if $s(w+\rho)$ contains a $0$ for some $s\in \W$.
 
If $\pi = (\pi_{1},\dots,\pi_{r})$ is a decreasing sequence of integers, and $w(\pi)$ is not dominant, Bott's algorithm says that either $w(\pi)$ is singular, in which case 
 \[H^{\bullet}(\Gr(r,A^{*}), S_{\pi}R^{*}) = 0\] 
 or
 \[
 H^{l(\omega)}(\Gr(r,A^{*}), S_{\pi}R^{*}) = S_{\mathcal{P}(\omega(w(\pi)+\rho)-\rho)}A,\]
where if $w$ is a weight, $\mathcal{P}(w)$ denotes the partition associated to $w$,
 $\omega$ is a minimal element of $\W$ such that $\omega(w(\pi)+\rho)-\rho$ is dominant, and $l(\omega)$ is its length.

 We will need the following straightforward applications of  Bott's algorithm.
\begin{prop}\label{prop:sym} The bundle $S^{q}\Q^{*} \o S^{p} \R$  over the Grassmannian $Gr(a-1,A^{*})$
has only the following cohomology:
\[\begin{array}{rll}
H^{0}(Gr(a-1,A),  S^{p} \R \o S^{q}\Q^{*} ) &=
(\bw{a-1}A)^{p-q}  \cong S^{p-q}A^{*} \o (\bw{a} A)^{p-q}
&\text{ if }\quad p-q\geq 0,\\
 H^{a-1}(Gr(a-1,A),  S^{p} \R \o S^{q}\Q^{*} )& =
  S^{q-p - a-2 }A &\text{ if }\quad  q-p - a-2 \geq 0,
\end{array}
\]
and all other cohomology vanishes.
\end{prop}
\begin{proof}[Sketch of proof]
Since $\Q$ is a line bundle, the weight of $ S^{p} \R \o S^{q}\Q^{*} $ is $[0,\dots0,p-q]$.  Let $s_{j}$ denote the affine reflection at node $j$.
The results of successively applying reflections are the following:
\[\begin{array}{rcl}
s_{a}([0,\dots,0,p-q]) &=&  [0, \dots0,0, p-q+1, -p+q-2]\\
s_{a-1}s_{a}([0,\dots,0,p-q])  &=& [0, \dots0,p-q+2, -p+q-3, 0]\\
&\vdots \\
s_{a-(j-1)}\dots s_{a}([0,\dots,0,p-q]) &=&[0,\dots,(p-q+j), (-p+q - j - 1),0\dots,0]\\
&\vdots \\
s_{1}s_{2}\dots s_{a-1}s_{a}([0,\dots,0,p-q]) &=& [-p+q-a-2,0 \dots0, 0]
.\end{array}
\]
If this weight after $j$ reflections (with  $0<j<a$) is dominant, then 
\[
p-q+j\geq 0 \quad \text{and} \quad -p+q-j-1 \geq 0
,\]
which implies the impossible
\[
j+1 \leq -p+q\leq j
.\]
 So, the only way to get nontrivial cohomology is when either the original weight $[0,\dots,0,p-q]$ is dominant (giving $H^{0}$) or the last weight $[-p+q-a-2]$ is dominant (giving $H^{a-1}$). 
\end{proof}
 
We can give a simple criteria that guarantees when a bundle $S_{\pi}\R$ is acyclic. Let $w = w(\pi)$, ignore the first $r$ entries and apply simple reflections.  We summarize the consequence of each reflection in the following table:
\[\tiny{
\begin{array}{rcl | |c}
&\text{weight}&& \text{condition to be non-singular} \\
\hline
w+\rho &=&[...,\pi_{r}+1,1,1,\dots,1] &\pi_{r}+1\notin \{0,-1\} \\
s_{r}(w+\rho) &=& [...,-\pi_{r}-1,\pi_{r}+2,1,\dots,1] &\pi_{r}+2\notin \{0,-1\}\\
s_{r+1}s_{r}(w+\rho) &=& [...,1,-\pi_{r}-2,\pi_{r}-3,1,\dots,1]&\pi_{r}+3\notin \{0,-1\}\\
&\vdots && \vdots\\
s_{a-3}\dots s_{r}(w+\rho) &=& [...,1,1,\dots,1,-\pi_{r}-(a-r-1)+1 ,  \pi_{r}+(a-r)-1,1] & \pi_{r}+a-r-1\notin \{0,-1\}\\
s_{a-2}\dots s_{r}(w+\rho) &=& [...,1,1,\dots,1,-\pi_{r}-(a-r)+1 ,  \pi_{r}+(a-r)] & \pi_{r}+a-r\neq 0
\end{array}}\]

So, if $\pi_{r}+(a-r)<0$, the weight is non-singular, and we can have higher cohomology (still depending on the rest of the partition).   But if $\pi_{r}+(a-r)\geq 0$, (\emph{i.e.} $\pi_{r}\geq -a+r$), then the weight is singular, and there can be no higher cohomology.

We have proved the following: (which be believe corrects Lemma~5.2(1) of \cite{Landsberg-Weyman-Bull07}, which had an $\pi_{1}\geq (-a+1)$ instead of $\pi_{r}\geq (-a+r)$, and we believe this is the statement they actually use in their proof).

\begin{lemma} Consider a partition $\pi=(\pi_{1},\dots,\pi_{r})$. If  $\pi_{r}\geq -a+r$,
 then the bundle $S_{\pi}\R^{*}$ over $\Gr(r,A^{*})$ is acyclic.
\end{lemma}

In the multiple factor case, for each $j$ let $\pi^{j}=(\pi^{j}_{1},\dots,\pi^{j}_{r})$ be non-increasing sequences of integers and consider the bundle $S_{\pi^{1}}\R_{1}\o\dots\o S_{\pi^{n}}\R_{n}$ over 
$\Gran$. Since a singular weight in any factor will cause the whole bundle to be singular, and every sheaf over $\PP^{0}$ is acyclic,
 we have the following extension:
\begin{lemma}\label{lem:top}
Consider $\pi^{j}=(\pi^{j}_{1},\dots,\pi^{j}_{r})$ and suppose that for every $j$ for which $0 <r_{j} <a_{j}$ we have $\pi^{j}_{r} \geq -a_{j}+r_{j}$.
Then  the bundle $S_{\pi^{1}}\R_{1}^{*}\o\dots\o S_{\pi^{n}}\R_{n}^{*}$ over $\Gr(r_{1},A_{1}^{*})\times \dots\times \Gr(r_{n},A_{n}^{*})$ is acyclic.
\end{lemma}

Now we slightly generalize and add to \cite[Lem.~5.2]{Landsberg-Weyman-Bull07} (and  \cite[Lem.~5.5]{LanWey_Tan}) for the unbalanced case. Our version of this lemma will allow us to work inductively.  

  We will use the following notation: $\w{r_{j}} = r_{1}r_{2}\cdots r_{n}/r_{j}$, and similarly $\w{a_{j}} = a_{1}\dots a_{n}/a_{j}$. Recall we have defined $\eta := \Rn$
 and $\xi := (\AA \o \O_{B} /\eta^{*})^{*}$.
\begin{lemma}\label{lem:MCMmodule} Let $B=\Gran$, let $\eta = \Rn$ and
 consider the sheaf of $\Sym(\eta)$-modules
\[\M = \bigotimes_{j}S_{\pi^{j}}\R_{j}^{*}.\] 
Assume  for all $1\leq j\leq n$  that $\pi^{j}_{1}\geq 0$ and for every $j$ for which $0<r_{j} <a_{j}$ we have $\pi^{j}_{1}\leq \w{r_{j}}-r_{j}$.
Then the $\Sym(A_{1}\o\dots\o A_{n})$-module $H^{0}(B,\M)$, which is supported on $\Sub_{r_{1},\dots,r_{n}}(A_{1}^{*}\o\dots\o A_{n}^{*})$, is a maximal Cohen-Macaulay module.

Moreover, for such $\M$ the last term of the resolution of $H^{0}(B,\M)$ is
\[H^{0}(B,\M^{\vee}) = \bigotimes_{j}
(S_{
\underbrace{\w{r_{j}}-a_{j},\dots,\w{r_{j}}-a_{j}}_{a_{j}-r_{j}},
\underbrace{
\pi^{j}_{1}
,\dots,
\pi^{j}_{r_{j}}
}_{r_{j}}
}  A_{j} ) 
\otimes (\bw{a_{j}}A_{j})^{\w{a_{j}} +a_{j} - (\w{r_{j}}+r_{j})  }
\]
In particular, $\M$ has an $(s_{j})$-small resolution, with $s_{j} = 
\begin{cases}
 \w{a_{j}}-r_{j}, & \text{ if } r_{j}<a_{j} \\
 \w{a_{j}} - \w{r_{j}}+\pi^{j}_{1}, & \text{ if }r_{j}=a_{j},
\end{cases}$
\end{lemma}

\begin{remark}
If $S_{\lambda^{1}}\R_{1}\o\dots\o S_{\lambda^{n}}\R_{n}$ is a module occurring in the minimal free resolution of $\M$ then we have $\lambda^{j}_{1} \leq  \w{a_{j}}-r_{j}$ if $r_{j}<a_{j}$, or 
$\lambda^{j}_{1}\leq  \w{a_{j}} - \w{r_{j}}+\pi^{j}_{1}$ if $r_{j}=a_{j}$.
In fact, the partition $\lambda^{j}$ must be obtained by deleting boxes from one of the partition in the last module in the resolution of $\M$, i.e. $\lambda^{j}$ is dominated by at least one of the partitions in the last module of the resolution of $\M$.
\end{remark}

\begin{proof}
This lemma is essentially \cite[Cor.~5.1.5]{Weyman} applied in our case, but as the colloquialism goes, ``the devil is in the details,'' so we carefully reproduce the necessary arguments. 
We repeat Landsberg and Weyman's proof almost verbatim in our case.
 This will allow us to get the more refined description of $H^{0}(B,\M)$ and its dual.

For any vector bundle $\V\to B$, \cite[Thm.~5.1.2]{Weyman} defines the complex
$F(\V)_{\bullet}$ 
\begin{equation}
F(\V)_{i} = \bigoplus_{d\geq 0}H^{d}(B,\bw{i+d}\xi \o \V )\o \Sym(\AAn )(-i-d).
\end{equation}
\begin{remark}\label{rmk:claim}
Of course $0\leq d\leq \dim B$, otherwise the cohomology vanishes. Also, $i+d \leq \rank\xi$ otherwise $\bw{i+d}\xi$ is zero. 
Therefore the top degree $i$ for which $F(\V)_{\bullet}$ is non-trivial is $i\leq \rank\xi-\dim B$.  The least degree $i$ for which $F(\V)_{\bullet}$ is non-zero is $i \geq -\dim B$.
The  \textbf{essential claim} is that  $F(\V)_{\bullet}$ does not have any non-trivial terms in negative degree. 
\end{remark}
If the claim in Remark~\ref{rmk:claim} holds, $F(\V)_{\bullet}$ must be a resolution of $H^{0}(B,\M)$ of length at most $\rank\xi-\dim B$, which is
the codimension of the support of $H^{0}(B,\M)$ by \cite[Thm.~5.1.6(a)]{Weyman}. By this dimension count, we know that $H^{0}(B,\M)$ is supported in $\SubA$.

To prove the claim, we compute the dual bundle $\M^{\vee}$ and its cohomology.
Let $K_{B}$ denote the canonical module for $B$.
%
Again for any vector bundle $\V\to B$,
 \cite[Thm.~5.1.4]{Weyman}, says that the twisted dual vector bundle
\[
\V^{\vee} = K_{B}\o \bw{\rank \xi }\xi^{*} \o \V^{*}
\]
satisfies
\[
F(\V^{\vee})_{s} = F(\V)^{*}_{s-(\rank \xi -\dim B )} . 
\]

In particular,
\begin{equation}
F(\V^{\vee})_{s} = \bigoplus_{d\geq 0}H^{d}(B,\bw{s+d}\xi \o  K_{B}\o \bw{\rank \xi }\xi^{*} \o \V^{*} )\o \Sym(\AAn )(-s-d).
\end{equation}
By Remark~\ref{rmk:claim}, the $s$ for which the complexes above are non-trivial is
$-\dim B\leq s \leq \rank\xi -\dim B$
and the crucial case is $s =0$:
\[
F(\V^{\vee})_{0} = F(\V)^{*}_{-(\rank\xi-\dim B)} 
.\]

Now we claim that the rightmost non-zero term in $F(\M^{\vee})_{\bullet}$ is $F(\M^{\vee})_{0}$.
Note the canonical module of the Grassmannian is isomorphic to $K_{\Gr(r,A^{*})} = S_{a-r,\dots,a-r}\R \o S_{r,\dots,r}Q^{*}$, which is unique up to a twist by a trivial line bundle.
Since $\u A^{*} = \R \oplus \Q$, and $\bw{a}\u A^{*} = S_{1,\dots,1}\R \o S_{1,\dots,1}\Q$, we have
\[
K_{\Gr(r,A^{*})}\cong K_{\Gr(r,A^{*})}\o (\bw{a}\u A^{*})^{r} = S_{a,\dots,a}\R,
\] 
or
\[
K_{\Gr(r,A^{*})} = S_{a,\dots,a}\R \o (\bw{a}\u A)^{r} = S_{a,\dots,a}\R \o  S_{r,\dots,r}\u A
.\] 

So on $B = \Gran$ we have
\begin{equation}\label{eq:KB}
K_{B} = \bigotimes_{j}K_{\Gr(r_{j},A_{j}^{*})} = \bigotimes_{j}S_{a_{j},\dots,a_{j}}\R_{j} \o S_{r_{j},\dots,r_{j}}\u A_{j}\\
= \bigotimes_{j}S_{-a_{j},\dots,-a_{j}}\R_{j}^{*}  \o S_{r_{j},\dots,r_{j}}\u A_{j}.
\end{equation}

\begin{remark}
Suppose $U$, $V$ and $W=U\oplus V$ are vector spaces of dimensions $u$, $v$ and $w =u+v$ respectively.
Then 
$\bw{w}W = \bw{u}U\o \bw{v}V$, and 
$\bw{w}W\o \bw{v}V^{*} = \bw{u}U\o \bw{v}V \o \bw{v}V^{*} = \bw{u}U$. So if we want to compute the top exterior power of the quotient $W/V$, we can think of $U$ as a quotient of $W$, and have $\bw{u}(W/V) = \bw{w}W\o \bw{v}V^{*}$.
\end{remark}

Since $\xi^{*} = \bw{a_{1}\dots a_{n}}(\u A_{1}^{*}\o \dots \o \u A_{n}^{*}) \o \O_{B}/\eta^{*}$ and $\eta = \R_{1}^{*}\o\dots\R_{n}^{*}$, we have
\[
\bw{\rank \xi} \xi^{*} = \left(  \bw{a_{1}\dots a_{n}}(\u A_{1}^{*}\o \dots \o \u A_{n}^{*}) \right)
\o \bw{r_{1}\dots r_{n}}\eta
=  \left( \bw{a_{1}\dots a_{n}}(\u A_{1}^{*}\o \dots \o \u A_{n}^{*}) \right)\o
\left(\bigotimes_{j}
S_{\w{r_{j}},\ldots,\w{r_{j}}}\R_{j}^{*} \right)
\]
\begin{equation}\label{eq:topxi}
\bw{\rank \xi} \xi^{*} =\bigotimes_{j}  S_{\underbrace{\w{a_{j}},\dots,\w{a_{j}}}_{a_{j}} } \u A_{j}^{*}\o
S_{\underbrace{\w{r_{j}},\ldots,\w{r_{j}}}_{r_{j}}  } \R_{j}^{*}  = 
.\end{equation}
 Up to tensoring with the trivial bundle on $B$ we can express this as
\begin{equation}\label{eq:xitop}
\bw{\rank \xi} \xi^{*} \cong \bigotimes_{j}
S_{\underbrace{\w{r_{j}},\ldots,\w{r_{j}}}_{r_{j}} }\R_{j}^{*}
.\end{equation}

Recall the sheaf of $\Sym(\eta)$-modules $\M = \bigotimes_{j}S_{\pi^{j}}\R_{j}^{*}$  with $\pi^{j}=(\pi_{1}^{j},\dots,\pi_{r_{j}}^{j})$, (by \eqref{eq:dualdual} ) has vector space dual $\M^{*} = \bigotimes_{j}S_{-\overleftarrow{\pi^{j}}}\R_{j}^{*}$.
  Take the tensor product of $\M^{*}$ with $K_{B}$ from \eqref{eq:KB} and $\bw{\rank \xi}\xi$ from \eqref{eq:xitop} to 
 compute 
$\M^{\vee}= K_{B}\o \bw{\rank \xi }\xi^{*} \o \M^{*}
$, which becomes,
\begin{equation}\label{eq:Mdual}
\M^{\vee} = \bigotimes_{j}
S_{\w{r_{j}}-a_{j}-\pi^{j}_{r_{j}},\dots, \w{r_{j}}-a_{j}-\pi^{j}_{1}}  \R_{j}^{*} \o S_{\w{a_{j}}-r_{j},\dots, \w{a_{j}}-r_{j}}\u A_{j}^{*},
\end{equation}
or, up to tensoring with a trivial line bundle,
\begin{equation}\label{eq:Mdual2}
\M^{\vee} \cong \bigotimes_{j}
S_{\w{r_{j}}-a_{j}-\pi^{j}_{r_{j}},\dots, \w{r_{j}}-a_{j}-\pi^{j}_{1}}  \R_{j}^{*}
.\end{equation}

We claim that the sheaf $\M^{\vee}$ has no higher cohomology.
Indeed, the last part on the $j$-th partition associated to $\M^{\vee}$ is $\w{r_{j}}-a_{j}-\pi^{j}_{1}$, and our hypothesis that $\pi^{j}_{1}\leq \w{r_{j}}-r_{j}$ implies that $\w{r_{j}}-a_{j}-\pi^{j}_{1} \geq -a_{j}+r_{j} $ so Lemma~\ref{lem:top} shows that $\M^{\vee}$ is acyclic. Moreover, any representation occurring in $\bw{i+d}\xi \o\M$ must have the last parts of its respective partitions at least as large as those in the corresponding factors for $\M$ (it's more ample), so, again by Lemma~\ref{lem:top}, $H^{d}(B,\bw{i+d}\xi \o\M^{\vee})=0$ for all $i+d> 0$.

Therefore the complexes $F(\M)_{\bullet}$ and $F(\M^{\vee})_{\bullet}$ have length equal to the codimension of the 
subspace variety $\SubA$, which is $\rank \xi - \dim B$, establishing the claim.

Moreover, the last term in the complex $F(\M)_{\bullet}$ is
\[\begin{array}{rcl}
F(\M)_{\rank{\xi}-\dim B} &=& F(\M^{\vee})_{0}^{*}
\\
H^{\dim B}(B, \bw{\rank \xi} \xi \o \M) 
&=& H^{0}(B,\M^{\vee})^{*} 
,\end{array}
\quad \text{up to a twist.}
\]


Using the expression for $\M^{\vee}$ \eqref{eq:Mdual2}, we have:
\[
\begin{array}{rcl}
H^{0}(B,\M^{\vee})^{*} 
&=&
H^{0}(B,\bigotimes_{j}
S_{\underbrace{ \w{r_{j}}-a_{j}-\pi^{j}_{r_{j}},\dots, \w{r_{j}}-a_{j}-\pi^{j}_{1}}_{r_{j}}}  \R_{j}^{*} 
)^{*}
\\
&=&\bigotimes_{j}
(S_{\underbrace{\w{r_{j}}-a_{j}-\pi^{j}_{r_{j}},\dots, \w{r_{j}}-a_{j}-\pi^{j}_{1}}_{r_{j}},\underbrace{0,\dots,0}_{a_{j}-r_{j}}}  A_{j} ) ^{*}
\\
&=&\bigotimes_{j}
(S_{
\underbrace{0,\dots,0}_{a_{j}-r_{j}},
\underbrace{
-(\w{r_{j}}-a_{j}-\pi^{j}_{1})
,\dots,
-(\w{r_{j}}-a_{j}-\pi^{j}_{r_{j}})
}_{r_{j}}
}  A_{j} ) \quad\text{(by the analog of \eqref{eq:dualdual})}
\\
&\cong& \bigotimes_{j}
(S_{
\underbrace{\w{r_{j}}-a_{j},\dots,\w{r_{j}}-a_{j}}_{a_{j}-r_{j}},
\underbrace{
\pi^{j}_{1}
,\dots,
\pi^{j}_{r_{j}}
}_{r_{j}}
}  A_{j} ) \quad\text{(up to a twist).}
\end{array}
\]

Now, to get the correct twist, we notice that the representations in the last term of $F(\M)$ must occur in the tensor product of the $\Sym(\AAn)$ module associated to $\M$ with $\O_{\AAn}(-\rank\xi )$.  The previous result establishes the shape up to a twist. The total degree of each factor must be the same, and equal to $|\pi^{j}| + \rank \xi = |\pi^{j}|+a_{1}\cdots a_{n} - r_{1}\cdots r_{n}$. 
the total degree of 
\[
 \bigotimes_{j}
(S_{
\underbrace{\w{r_{j}}-a_{j},\dots,\w{r_{j}}-a_{j}}_{a_{j}-r_{j}},
\underbrace{
\pi^{j}_{1}
,\dots,
\pi^{j}_{r_{j}}
}_{r_{j}}
}  A_{j} ) 
\]
is 
$(a_{j}-r_{j})(\w{r_{j}} -a_{j}  ) + |\pi^{j}|$.
To make the total degree in each factor equal to
$ |\pi^{j}|+a_{1}\cdots a_{n} - r_{1}\cdots r_{n}$,
we twist by 
\[(\bw{a_{j}}A_{j})^{(\w{a_{j}}  +a_{j}- (\w{r_{j}}+ r_{j}))}.\]

%
After some simplification, the last module in the resolution must be equal to
%
\[
\bigotimes_{j}
(S_{
\underbrace{
\w{a_{j}} -r_{j}
,\dots,
\w{a_{j}} -r_{j}
}_{a_{j}-r_{j}},
\underbrace{
\w{a_{j}}  +a_{j}- (\w{r_{j}}+ r_{j})+\pi^{j}_{1}
,\dots,
\w{a_{j}}  +a_{j}- (\w{r_{j}}+ r_{j})+\pi^{j}_{r_{j}}
}_{r_{j}}
}  A_{j} ) 
\]


So the first part of a partition in a representation in $F(\M)_{\rank\xi -\dim B}$ is either
$ \w{a_{j}}-r_{j}$ in the factors where $r_{j}<a_{j}$, or, 
in the factors where
$r_{j} = a_{j}$, the module appearing in the $j$-th factor of $F(\M)_{\rank\xi -\dim B}$ is simply $S_{\pi^{j}}A_{j}\o (\bw{a_{j}}A_{j})^{\w{a_{j}} - \w{r_{j}}}$, so the first part of the partition is $ \w{a_{j}} - \w{r_{j}}+\pi^{j}_{1}$.
\end{proof}

\section{Mapping cone construction}\label{sec:cone}

The essential argument used to prove Theorem~\ref{thm:workhorse} is to construct a mapping cone, and use the resolution of the subspace variety to lift a resolution of the smaller variety to get a possibly non-minimal resolution of the larger one.  We do this by applying a mapping cone construction together with Lemma~\ref{lem:MCMmodule}.

\begin{proof}[Proof of Theorem~\ref{thm:workhorse}]
The proof is almost identical to Landsberg and Weyman's original proof, and might be safely omitted,
but to show our slight improvement, we provide all of the details.

Recall the partial desingularization from above
\[
\xymatrix{
\R_{1}\o \R_{2}\o \dots\o \R_{n} 
\ar[d]^{p} \ar[dr]^{q}\\
\Gran & \overline{G.\w{Y}} &\subset& \SubA
,}\]
Again denote by $\tilde Z$  the total space of the bundle $\R_{1}\o \dots \o \R_{n}$ with 
each fiber over $(A'^{*}_{1},\dots,A'^{*}_{n}) \in \Gran$ equal to $\AAp$. 
Define bundles $\eta := \R_{1}^{*}\o\dots\o \R_{n}^{*}$
 and $\xi := (\AA \o \O_{B} /\eta^{*})^{*}$,
and groups $G:= \GLA$ and $G':=\GLAp$.

By hypothesis, $Y$ has a $G'$-equivariant resolution $E_{\bullet}$, with $E_{0}= {\bf A}:=\Sym(\AA)$
 and $E_{1}$ is the ideal of $Y$ in the particular fiber;  
 whose length is equal to the codimension of $Y$ inside $\PP(\AAp)$ (we assume that $Y$ is aCM).


Let ${\bf B}:=\Sym(\eta)$ (by Weyman 5.1.1b), which is a sheaf of algebras isomorphic to $p_{*}(\O_{\tilde Z})$.  We form a complex of sheaves of ${\bf B}$-modules from $E_{\bullet}$ by replacing $E_{i}$ with the sheaf $\E_{i}$ obtained by replacing the Schur functors of vector spaces $A_{j}^{*}$ with the corresponding Schur functors on the sheaves $\R_{j}$.

We have projections $q:Z\to Y$ and $p: Z\to \Gran$. Also $p_{*}(\O_{Z}) = {\bf B}/d(\E_{1})$ as $d(\E_{1})$ is the sub sheaf of ${\bf B}$ of local functions on $\tilde{Z}$ that vanish on $Z$.

Our complex of sheaves of ${\bf B}$-modules $\E_{\bullet}$ is such that each term is a sum of terms of the form
\[S_{\pi^{1}}\R_{1}\o\dots\o S_{\pi^{n}}\R_{n},\]
and each term is homogeneous and completely reducible, with each irreducible summand having nonzero $H^{0}$, so in particular, no term has any higher cohomology.

Define a complex $M_{\bullet}$ of ${\bf A}$-modules by letting $M_{j}:=H^{0}(\Gran, \E_{j})$.
The minimal free resolution of the ideal of $\overline{G.Y}$ is the minimal resolution of the cokernel of the complex $M_{\bullet}$: Notice, the cokernel $M_{0}/\text{Image}(M_{1})$ is exactly $\CC[X]$ because $M_{0}$ consists of functions on the subspace variety and $M_{1}$ the ideal of $Y$ inside the subspace variety.

To obtain a non-necessarily minimal resolution of $\CC[\overline{G.\w{Y}}]$, iterate the mapping cone construction as follows:

Let $F_{j,\bullet}$ be a resolution of $M_{j}$ for each $j$. Obtain a double complex:
\[
\xymatrix{
\ar[d] &  \ar[d]&&\ar[d]\\
F_{L,1}\ar[d]\ar[r]& F_{L-1,1} \ar[d]\ar[r]& \dots \ar[r]& F_{0,1}\ar[d]\\
F_{L,0}\ar[d]\ar[r]& F_{L-1,0} \ar[d]\ar[r]& \dots \ar[r]& F_{0,0}\ar[d]\\
M_{L}\ar[r]& M_{L-1} \ar[r]& \dots \ar[r]& M_{0} \ar[r]& \mathcal{C}
}
\]
Sum the SW to NE diagonals to get a complex with terms $\tilde{F}_{j} = \bigoplus_{s+t=j} F_{s,t}$.  And we have that $\tilde{F}_{\bullet}$ is a resolution of $\mathcal{C}= \CC[\overline{G.\w{Y}}]$.

Now, by Lemma~\ref{lem:MCMmodule} the modules $M_{i}$ are maximal Cohen-Macaulay; hence the lengths of the minimal free resolutions all equal codim $\SubA$
which is equal to $\rank \xi - \dim (\Gran)$.

The complexes $F_{i,\bullet}$ are resolutions of the $M_{i}$, with maximum length equal to the codimension of $\SubA$, so the number of diagonals in our rectangular bi-complex is at most
\[
\codim\SubA + L
,\]
where $L$ is the codimension of $Y$ in $\PP(\AA)$.
But the codimension of $\overline{G.\w{Y}}$ is the codimension of $Y$ plus the codimension of $\SubA$, by construction.
So the length of the possibly non-minimal resolution of $\CC[\overline{G.\w{Y}}]$  is at most $\codim\SubA + L$, which is its codimension.

Finally, we get the statement about the small-ness of the resolution obtained by applying the ``moreover'' part of Lemma~\ref{lem:MCMmodule}.
\end{proof}

\end{appendix}
\section*{Acknowledgements}
The author would like to thank Mateusz Michalek, Claudiu Raicu, and Steven Sam for useful discussions.
\bibliographystyle{amsplain}
\newcommand{\arxiv}[1]{{\tt \href{http://arxiv.org/abs/#1}{{arXiv:#1}}}}
\bibliography{/Users/oeding/Dropbox/bibtex_bib_files/main_bibfile}

\def\Dbar{\leavevmode\lower.6ex\hbox to 0pt{\hskip-.23ex \accent"16\hss}D}
  \def\cprime{$'$}
\providecommand{\bysame}{\leavevmode\hbox to3em{\hrulefill}\thinspace}
\providecommand{\MR}{\relax\ifhmode\unskip\space\fi MR }
\providecommand{\MRhref}[2]{%
  \href{http://www.ams.org/mathscinet-getitem?mr=#1}{#2}
}
\providecommand{\href}[2]{#2}
\begin{thebibliography}{10}

\bibitem{AOP_Segre}
H.~Abo, G.~Ottaviani, and C.~Peterson, \emph{Induction for secant varieties of
  {S}egre varieties}, Trans. Amer. Math. Soc. \textbf{361} (2009), no.~2,
  767--792.

\bibitem{AOP_Grassmann}
\bysame, \emph{Non-defectivity of {G}rassmannians of planes}, J. Algebraic
  Geom. \textbf{21} (2012), no.~1, 1--20.

\bibitem{AholtOeding}
C.~{Aholt} and L.~{Oeding}, \emph{{The ideal of the trifocal variety}}, Math.
  Comp. \textbf{83} (2014), no.~289, 2553--2574, \arxiv{1205.3776}.

\bibitem{AH95}
J.~Alexander and A.~Hirschowitz, \emph{Polynomial interpolation in several
  variables}, J. Algebraic Geom. \textbf{4} (1995), no.~2, 201--222.

\bibitem{prize}
E.~Allman, \emph{Open problem: Determine the ideal defining
  {$Sec_{4}(\mathbb{P}^{3} \times \mathbb{P}^{3} \times \mathbb{P}^{3})$}},
  {http://www.dms.uaf.edu/~eallman/salmonPrize.pdf}, 2010.

\bibitem{AllmanRhodes08}
E.~Allman and J.~Rhodes, \emph{Phylogenetic ideals and varieties for the
  general {M}arkov model}, Adv. in Appl. Math. \textbf{40} (2008), no.~2,
  127--148.

\bibitem{JMLR:v15:anandkumar14a}
A.~Anandkumar, R.~Ge, D.~Hsu, and S.~Kakade, \emph{A tensor approach to
  learning mixed membership community models}, Journal of Machine Learning
  Research \textbf{15} (2014), 2239--2312.

\bibitem{AshHillarFinite}
M.~Aschenbrenner and C.~J. Hillar, \emph{Finite generation of symmetric
  ideals}, Trans. Amer. Math. Soc. \textbf{359} (2007), no.~11, 5171--5192.

\bibitem{Ballico2005_weak}
E.~Ballico, \emph{On the weak non-defectivity of {V}eronese embeddings of
  projective spaces}, Cent. Eur. J. Math. \textbf{3} (2005), no.~2, 183--187
  (electronic).

\bibitem{Bertini}
D.~J. Bates, J.~D. Hauenstein, A.~J. Sommese, and C.~W. Wampler, \emph{Bertini:
  software for numerical algebraic geometry}, Available at
  \url{bertini.nd.edu}, 2006.

\bibitem{BatesOeding}
D.~J. Bates and L.~Oeding, \emph{Toward a salmon conjecture}, Exp. Math.
  \textbf{20} (2011), no.~3, 358--370, \arxiv{1009.6181}.

\bibitem{Bernardi_Ideal_Sym}
A.~Bernardi, \emph{Ideals of varieties parameterized by certain symmetric
  tensors}, J. Pure Appl. Algebra \textbf{212} (2008), no.~6, 1542--1559.

\bibitem{BocciChiantini}
C.~Bocci and L.~Chiantini, \emph{On the identifiability of binary {S}egre
  products}, J. Algebraic Geom. \textbf{22} (2013), no.~1, 1--11.

\bibitem{BocciChiantiniOttaviani}
C.~Bocci, L.~Chiantini, and G.~Ottaviani, \emph{Refined methods for the
  identifiability of tensors}, Ann. Mat. Pura Appl. (4) \textbf{193} (2014),
  no.~6, 1691--1702.

\bibitem{OttBra08_AH}
C.~Brambilla and G.~Ottaviani, \emph{On the {A}lexander-{H}irschowitz theorem},
  J. Pure Appl. Algebra \textbf{212} (2008), no.~5, 1229--1251.

\bibitem{BucLan_ranks}
J.~Buczy{\'n}ski and J.M. Landsberg, \emph{Ranks of tensors and a
  generalization of secant varieties}, Linear Algebra Appl. \textbf{438}
  (2013), no.~2, 668--689.

\bibitem{CEO}
D.~Cartwright, D.~Erman, and L.~{Oeding}, \emph{Secant varieties of
  $\mathbb{P}^{2}\times \mathbb{P}^{n}$ embedded by $\mathcal{O}(1,2)$}, J.
  Lond. Math. Soc. (2) \textbf{85} (2012), no.~1, 121--141, \arxiv{1009.1199}.

\bibitem{CGG4_Segre}
M.~V. Catalisano, A.~V. Geramita, and A.~Gimigliano, \emph{Higher secant
  varieties of the {S}egre varieties {$\Bbb P\sp 1\times\dots\times\Bbb P\sp
  1$}}, J. Pure Appl. Algebra \textbf{201} (2005), no.~1-3, 367--380.

\bibitem{CGG5_rational}
\bysame, \emph{On the ideals of secant varieties to certain rational
  varieties}, J. Algebra \textbf{319} (2008), no.~5, 1913--1931.

\bibitem{CGG_P1s}
\bysame, \emph{Secant varieties of {$\Bbb P^1\times\dots\times\Bbb P^1$}
  ({$n$}-times) are not defective for {$n\geq 5$}}, J. Algebraic Geom.
  \textbf{20} (2011), no.~2, 295--327.

\bibitem{ChiCil06_ksecant}
L.~Chiantini and C.~Ciliberto, \emph{On the concept of {$k$}-secant order of a
  variety}, J. London Math. Soc. (2) \textbf{73} (2006), no.~2, 436--454.

\bibitem{ChiOttVan}
L.~Chiantini, G.~Ottaviani, and N.~Vannieuwenhoven, \emph{An algorithm for
  generic and low-rank specific identifiability of complex tensors}, SIAM J.
  Math. Anal. \textbf{35} (2014), no.~4, 1265--1287.

\bibitem{CGLM_Waring}
P.~Comon, G.~Golub, L.-H. Lim, and B.~Mourrain, \emph{Symmetric tensors and
  symmetric tensor rank}, SIAM J. Matrix Anal. Appl. \textbf{30} (2008), no.~3,
  1254--1279.

\bibitem{Kac85}
J.~Dadok and V.~Kac, \emph{Polar representations}, J. Algebra \textbf{92}
  (1985), no.~2, 504--524.

\bibitem{DaleoHauenstein}
N.~S. {Daleo} and J.~D. {Hauenstein}, \emph{{Numerically deciding the
  arithmetically {C}ohen-{M}acaulayness of a projective scheme}}, Journal of
  Symbolic Computation \textbf{72} (2016), 128--146.

\bibitem{lathauwer:642_matrix_diag}
L.~De~Lathauwer, \emph{A link between the canonical decomposition in
  multilinear algebra and simultaneous matrix diagonalization}, SIAM Journal on
  Matrix Analysis and Applications \textbf{28} (2006), no.~3, 642--666.

\bibitem{deSilva_Lim_ill_posed}
V.~de~Silva and L.-H. Lim, \emph{Tensor rank and the ill-posedness of the best
  low-rank approximation problem}, SIAM J. Matrix Anal. Appl. \textbf{30}
  (2008), 1084--1127.

\bibitem{DolgachevAG}
I.~V. Dolgachev, \emph{Classical algebraic geometry}, Cambridge University
  Press, Cambridge, 2012, A modern view.

\bibitem{draisma-kuttler_bounded}
J.~Draisma and J.~Kuttler, \emph{Bounded-rank tensors are defined in bounded
  degree}, Duke Math. J. \textbf{163} (2014), no.~1, 35--63, \arxiv{1103.5336}.

\bibitem{Eisenbud_tome}
D.~Eisenbud, \emph{{Commutative Algebra}}, Graduate Texts in Mathematics, vol.
  150, Springer-Verlag, New York, 1995, With a view toward algebraic geometry.

\bibitem{Eisenbud_syzygies}
\bysame, \emph{The geometry of syzygies}, Graduate Texts in Mathematics, vol.
  229, Springer-Verlag, New York, 2005, A second course in commutative algebra
  and algebraic geometry.

\bibitem{Fisher}
Tom Fisher, \emph{Pfaffian presentations of elliptic normal curves}, Trans.
  Amer. Math. Soc. \textbf{362} (2010), no.~5, 2525--2540. \MR{2584609}

\bibitem{Friedland2010_salmon}
S.~Friedland, \emph{On tensors of border rank {$l$} in {$\Bbb{C}^{m\times
  n\times l}$}}, Linear Algebra Appl. \textbf{438} (2013), no.~2, 713--737.

\bibitem{Friedland-Gross2011_salmon}
S.~Friedland and E.~Gross, \emph{A proof of the set-theoretic version of the
  salmon conjecture}, J. Algebra \textbf{356} (2012), 374--379.

\bibitem{HighestWeights}
F.~Galetto, \emph{{HighestWeights}: a package for free resolutions and modules
  with a semisimple lie group action}, preprint (2014),
  \url{http://www.mast.queensu.ca/~galetto/research.htm}.

\bibitem{Geramita_Lectures}
A.~V. Geramita, \emph{Inverse systems of fat points: {W}aring's problem, secant
  varieties of {V}eronese varieties and parameter spaces for {G}orenstein
  ideals}, The {C}urves {S}eminar at {Q}ueen's, {V}ol.\ {X} ({K}ingston, {ON},
  1995), Queen's Papers in Pure and Appl. Math., vol. 102, Queen's Univ.,
  Kingston, ON, 1996, pp.~2--114.

\bibitem{Ginensky}
A.~Ginensky, \emph{A generalization of the {C}lifford index and determinantal
  equations for curves and their secant varieties}, ProQuest LLC, Ann Arbor,
  MI, 2008, Thesis (Ph.D.)--The University of Chicago.

\bibitem{Harris}
J.~Harris, \emph{Algebraic geometry: a first course (graduate texts in
  mathematics)}, Graduate Texts in Mathematics, vol. 133, New York:
  Springer-Verlag, 1992, A first course.

\bibitem{HOOS}
J.~D. Hauenstein, L.~Oeding, G.~Ottaviani, and A.~Sommese, \emph{{Homotopy
  techniques for tensor decomposition and perfect identifiability}}, preprint
  (2014), \arxiv{1501.00090}.

\bibitem{IarrobinoKanev_text}
A.~Iarrobino and V.~Kanev, \emph{Power sums, {G}orenstein algebras, and
  determinantal loci}, Lecture Notes in Mathematics, vol. 1721,
  Springer-Verlag, Berlin, 1999, Appendix C by A. Iarrobino and S. Kleiman.

\bibitem{Kac80}
V.~Kac, \emph{Some remarks on nilpotent orbits}, J. Algebra \textbf{64} (1980),
  no.~1, 190--213.

\bibitem{Kanev_catalecticant}
V.~Kanev, \emph{Chordal varieties of {V}eronese varieties and catalecticant
  matrices}, J. Math. Sci. (New York) \textbf{94} (1999), no.~1, 1114--1125,
  Algebraic geometry, 9.

\bibitem{LandsbergSummary}
J.~M. Landsberg, \emph{Geometry and complexity theory}, Available at
  \url{http://www.math.tamu.edu/~jml/simonsclass.pdf}, 2014.

\bibitem{LandsbergTensorBook}
J.M. Landsberg, \emph{Tensors: geometry and applications}, Graduate Studies in
  Mathematics, vol. 128, American Mathematical Society, Providence, RI, 2012.

\bibitem{Landsberg-Manivel04}
J.M. Landsberg and L.~Manivel, \emph{On the ideals of secant varieties of
  {S}egre varieties}, Found. Comput. Math. \textbf{4} (2004), no.~4, 397--422.

\bibitem{LanMan08_Strassen}
\bysame, \emph{Generalizations of {S}trassen's equations for secant varieties
  of {S}egre varieties}, Comm. Algebra \textbf{36} (2008), no.~2, 405--422.

\bibitem{LanOtt11_Equations}
J.M. Landsberg and G.~Ottaviani, \emph{Equations for secant varieties of
  {V}eronese and other varieties}, Ann. Mat. Pura Appl. (4) (2011), 1--38
  (English), \arxiv{:1111.4567}.

\bibitem{LanWey_Tan}
J.M. Landsberg and J.~Weyman, \emph{On tangential varieties of rational
  homogeneous varieties}, J. Lond. Math. Soc. (2) \textbf{76} (2007), no.~2,
  513--530.

\bibitem{Landsberg-Weyman-Bull07}
\bysame, \emph{On the ideals and singularities of secant varieties of {S}egre
  varieties}, Bull. Lond. Math. Soc. \textbf{39} (2007), no.~4, 685--697.

\bibitem{LanWey_secant_CHSS}
\bysame, \emph{On secant varieties of compact {H}ermitian symmetric spaces}, J.
  Pure Appl. Algebra \textbf{213} (2009), no.~11, 2075--2086.

\bibitem{MaclaganSmith}
D.~Maclagan and G.~Smith, \emph{Multigraded castelnuovo-mumford regularity}, J.
  REINE ANGEW. MATH \textbf{571} (2003), 179--212, \arxiv{math/0305214}.

\bibitem{MichalekOedingZwiernik}
M.~{Michalek}, L.~{Oeding}, and P.~{Zwiernik}, \emph{Secant cumulants and toric
  geometry}, International Mathematics Research Notices \textbf{2015} (2015),
  no.~12, 4019--4063, \arxiv{1212.1515}.

\bibitem{Nie_STD}
J.~{Nie}, \emph{{Generating Polynomials and Symmetric Tensor Decompositions}},
  ArXiv e-prints (2014).

\bibitem{Nurmiev2}
A.~G. Nurmiev, \emph{Closures of nilpotent orbits of cubic matrices of order
  three}, Uspekhi Mat. Nauk \textbf{55} (2000), no.~2(332), 143--144.

\bibitem{Nurmiev}
\bysame, \emph{Orbits and invariants of third-order matrices}, Mat. Sb.
  \textbf{191} (2000), no.~5, 101--108.

\bibitem{OedOtt13_Waring}
L.~Oeding and G.~Ottaviani, \emph{Eigenvectors of tensors and algorithms for
  {W}aring decomposition}, J. Symbolic Comput. \textbf{54} (2013), 9--35,
  \arxiv{1103.0203}.

\bibitem{OedingSam}
L.~Oeding and S.~V Sam, \emph{{Equations for the fifth secant variety of Segre
  products of projective spaces}}, Experimental Mathematics \textbf{25} (2016),
  94--99, \arxiv{1502.00203}.

\bibitem{Ott07_Luroth}
G.~Ottaviani, \emph{Symplectic bundles on the plane, secant varieties and
  {L}\"uroth quartics revisited}, Vector bundles and low codimensional
  subvarieties: state of the art and recent developments, Quad. Mat., vol.~21,
  Dept. Math., Seconda Univ. Napoli, Caserta, 2007, pp.~315--352.

\bibitem{Pachter-Sturmfels}
L.~Pachter and B.~Sturmfels (eds.), \emph{Algebraic statistics for
  computational biology}, New York: Cambridge University Press, 2005.

\bibitem{Raicu_thesis}
C.~{Raicu}, \emph{Secant varieties of {S}egre-{V}eronese varieties}, ProQuest
  LLC, Ann Arbor, MI, 2011, Thesis (Ph.D.)--University of California, Berkeley.

\bibitem{RaicuGSS}
\bysame, \emph{Secant varieties of {S}egre-{V}eronese varieties}, Algebra
  Number Theory \textbf{6-8} (2012), 1817--1868.

\bibitem{RaicuCat}
C.~Raicu, \emph{{$3\times3$} minors of catalecticants}, Math. Res. Lett.
  \textbf{20} (2013), no.~4, 745--756, \arxiv{1011.1564}.

\bibitem{SahnounComon}
S.~Sahnoun and P.~Comon, \emph{{Tensor polyadic decomposition for antenna array
  processing}}, {CompStat'2014} (Geneva, Suisse) (The International Statistical
  Institute International~Association for Statistical~Computing, ed.), 2014
  (Anglais).

\bibitem{SamBounded}
S.~V Sam, \emph{The ideals of bounded rank symmetric tensors are generated in
  bounded degree}, preprint (2015), \arxiv{1510.04904}.

\bibitem{sam-snowden-tca}
S.~V Sam and A.~Snowden, \emph{{Introduction to twisted commutative algebras}},
   (2012), \arxiv{1209.5122}.

\bibitem{SamSnowden}
\bysame, \emph{Stability patterns in representation theory}, Forum Math. Sigma
  \textbf{3} (2015), e11, 108.

\bibitem{SidmanVermeire_Secant}
J.~Sidman and P.~Vermeire, \emph{Equations defining secant varieties: geometry
  and computation}, Combinatorial aspects of commutative algebra and algebraic
  geometry, Abel Symp., vol.~6, Springer, Berlin, 2011, pp.~155--174.

\bibitem{SidmanSullivant_Prolongations}
Jessica Sidman and Seth Sullivant, \emph{Prolongations and computational
  algebra}, Canad. J. Math. \textbf{61} (2009), no.~4, 930--949.

\bibitem{Strassen83_rank}
V.~Strassen, \emph{Rank and optimal computation of generic tensors}, Linear
  Algebra Appl. \textbf{52/53} (1983), 645--685.

\bibitem{SturmfelsSullivant}
B.~Sturmfels and S.~Sullivant, \emph{Combinatorial secant varieties}, Pure
  Appl. Math. Q. \textbf{2} (2006), no.~3, part 1, 867--891,
  \arxiv{math/0506223}.

\bibitem{Trung_degree_bounds}
N.~Trung and G.~Valla, \emph{Degree bounds for the defining equations of
  arithmetically {C}ohen-{M}acaulay varieties}, Math. Ann. \textbf{281} (1988),
  no.~2, 209--218.

\bibitem{Vinberg-Elasvili}
{\`E}.~B. Vinberg and A.~G. {\`E}la{\v{s}}vili, \emph{A classification of the
  three-vectors of nine-dimensional space}, Trudy Sem. Vektor. Tenzor. Anal.
  \textbf{18} (1978), 197--233.

\bibitem{Weyman}
J.~Weyman, \emph{Cohomology of vector bundles and syzygies}, Cambridge Tracts
  in Mathematics, vol. 149, Cambridge University Press, 2003.

\end{thebibliography}

\end{document}